\documentclass[11pt,reqno]{amsart} 
\usepackage[utf8]{inputenc}
\usepackage{amsfonts,amsmath,amsthm,amssymb}
\usepackage[justification=centering]{caption}
\usepackage{color}
\usepackage{fullpage}
\usepackage{enumitem}   
\usepackage{hyperref,cleveref, float}
\usepackage{mathtools}
\usepackage{xcolor}
\usepackage{thm-restate}
\usepackage{comment}
\usepackage{xurl}
\usepackage{subcaption}
\usepackage{thmtools}
\usepackage{hyperref}
\usepackage{cleveref}			% for links
\usepackage{dsfont}
\usepackage{enumitem}
\usepackage{tikz}
\usepackage{shuffle}
\usepackage{ifthen}
\usepackage{young}
\newtheorem{theorem}{Theorem}[section]
\newtheorem*{theorem*}{Theorem}
\newtheorem{proposition}[theorem]{Proposition}
\newtheorem{lemma}[theorem]{Lemma}
\newtheorem{corollary}[theorem]{Corollary}

\theoremstyle{definition}
\newtheorem{example}[theorem]{Example}
\newtheorem{definition}[theorem]{Definition}

\usepackage{etoolbox}
\AtEndEnvironment{proof}{\setcounter{claim}{0}}
\newtheoremstyle{notation}
{1em}
{1em}
{}
{}
{\bfseries}
{.}
{.5em} 
{} 
\theoremstyle{notation}

\newcommand{\Avoid}{\mathrm{Av}}
\newcommand*{\y}{\mathbf{y}}
\newcommand*{\x}{\mathbf{x}}

\newcommand{\Sym}{\mathfrak{S}}

\newcommand{\sylv}{\equiv_{sylv}}
\newcommand{\ssylv}{\equiv_{\#-sylv}}
\DeclareMathOperator{\std}{\text{std}}

\newcommand{\PF}{\textnormal{PF}}

\newcommand{\D}{\mathcal{D}}

\newcommand{\pf}{\mathsf{pf}}

\newcommand{\Des}{\text{Des}}
\newcommand{\iDes}{\text{iDes}}
\usetikzlibrary{patterns}
\usetikzlibrary{calc}
\usetikzlibrary{ decorations.markings}
\newcommand{\cross}[2]{ \draw[thick] (#1,#2)--(#1-1,#2-1);
     \draw[thick] (#1-1,#2)--(#1,#2-1);}
\newcommand{\rsquare}[2]{ \fill[pattern=north west lines, pattern color=gray] (#1,#2) rectangle ++(-1,-1);}

\newcommand{\cirw}[2]{ \draw[fill = white] (#1-0.5,#2-0.5) circle (.2);}

\newcommand{\shadrothe}[1]{
 \foreach \y [count=\x] in #1 {
     \foreach \i [count=\j] in #1{
         \ifthenelse{\x < \j}{\draw[ultra thick] (\j-1, \y-.5) -- (\j, \y-.5);}{}
         \ifthenelse{\i > \y}{\draw[ultra thick] (\x-.5, \i-1) -- (\x-.5, \i);}{}
    }
    \draw[ultra thick] (\x-.5, \y-.5) -- (\x-.5, \y);
    \draw[ultra thick] (\x-.5, \y-.5) -- (\x, \y-.5);
    \cirw{\x}{\y}
    }
}
\newcommand{\lgshadrothe}[2]{
 \foreach \y [count=\x] in #1 {
     \foreach \i [count=\j] in #1{
         \ifthenelse{\x < \j}{\draw[ultra thick] (\j-1, \y-.5) -- (\j, \y-.5);}{}
         \ifthenelse{\i > \y}{\draw[ultra thick] (\x-.5, \i-1) -- (\x-.5, \i);}{}
    }
    \draw[ultra thick] (\x-.5, \y-.5) -- (\x-.5, \y);
    \draw[ultra thick] (\x-.5, \y-.5) -- (\x, \y-.5);
    \cirw{\x}{\y}
    \draw (\x-.5, -.5) node {\x};
    \draw (-.5, \x-.5) node {\x};
    }
    \draw (0,0) grid (#2,#2);
}
\newcommand{\permrothe}[1]{
 \foreach \y [count=\x] in #1 {
     \cross{\x}{\y}
      \foreach \i [count=\j] in #1{
         \ifthenelse{\x < \j}{\ifthenelse{\y > \i}{\rsquare{\x}{\i}}{}}{}
         }
    }

}
\newcommand{\lgpermrothe}[2]{
 \foreach \y [count=\x] in #1 {
     \cross{\x}{\y}
      \foreach \i [count=\j] in #1{
         \ifthenelse{\x < \j}{\ifthenelse{\y > \i}{\rsquare{\x}{\i}}{}}{}
         }
    \draw (\x-.5, -.5) node {\x};
    \draw (-.5, \x-.5) node {\x};
    }
    \draw (0,0) grid (#2,#2);
}
\newcommand{\lgpermc}[2]{
 \foreach \y [count=\x] in #1 {
     \cross{\x}{\y}
    \draw (\x-.5, -.5) node {\x};
    \draw (-.5, \x-.5) node {\x};
    }
    \draw (0,0) grid (#2,#2);
}
\newcommand{\nwgrid}[2]{
\begin{scope}[decoration={markings, mark=at position 0.5 with {\arrow{>}}}] 
\coordinate (a) at (1,0);
\coordinate (b) at (0,1);
\foreach \row in {1,...,#2}{
\draw[postaction={decorate}] ($\row*(b)-(b)$)--($\row*(b)$);
\foreach \col  in {1,...,#1}{
\draw[postaction={decorate}] ($\row*(b)+\col*(a)-(b)$)--($\row*(b)+\col*(a)$);
\draw[postaction={decorate}] ($\row*(b)+\col*(a)$)--($\row*(b)+\col*(a)-(a)$);
}
}
\foreach \row  in {1,...,#1}{
\draw[postaction={decorate}] ($\row*(a)$)--($\row*(a)-(a)$);
}
\end{scope}
}
 \definecolor{new_green}{HTML}{009E73}
 \definecolor{new_pink}{HTML}{D35FB7}
 \definecolor{new_blue}{HTML}{0072B2}

\title{Enumerating Pattern Avoiding Parking Functions}
\author[Adenbaum]{Ben Adenbaum}
\address[B.~Adenbaum]{Department of Mathematics and Statistics, Villanova University, Villanova, PA, 19085}
\email{\textcolor{blue}{\href{mailto:benadenbaummath@gmail.com}{benadenbaummath@gmail.com}}}

\begin{document}

\begin{abstract}
In this paper, we complete the enumeration of the number of parking functions of length $n$ avoiding, in the sense defined by Qiu and Remmel, a permutation of length 3, answering several questions of Adeniran and Pudwell. Additionally, we provide explicit growth rates for the number of parking functions of length $n$ avoiding a monotonic pattern of any length. As a consequence of our techniques we additionally provide a combinatorial enumeration for the number of equivalence classes of words with a fixed content under the Sylvester congruence.
\end{abstract}
\maketitle
\section{Introduction}
Throughout we let $[n]=\{1,2,3,\ldots,n\}$ and denote the symmetric group on $[n]$ by $\Sym_n$. We write permutations in one-line notation $\pi=\pi_1\pi_2\cdots\pi_n$.
We recall that a permutation $\pi$ avoids a pattern $\sigma=\sigma_1\sigma_2\cdots\sigma_k$ if there is no subpermutation  of $\pi$ which is order isomorphic to $\sigma$.
Given a pattern $\sigma$, we let  $\Avoid_n(\sigma)$ to be the set of permutations in $\Sym_n$, which avoid the pattern $\sigma$.

One super set of permutations is parking functions. A tuple $p=(p_1,p_2,\ldots,p_n)\in[n]^n$ is a parking functions if and only if the weakly increasing rearrangement $p^\uparrow=(p_1',p_2',\ldots,p_n')$ satisfies $p_i'\leq i$ for all $i\in[n]$. 
For example, the tuple $(1,4,3,4)$ is not a parking function, while every permutation is a parking function. 
Parking functions receive their name as they can also be determined based on a parking process. The formal definition of parking functions is given in Definition~\ref{def:parking} together with a description of the parking process.
 We let $\PF_n$ denote the set of parking functions with $n$ cars and $n$ parking spots. Konheim and Weiss established that $|\PF_n|=(n+1)^{n-1}$ \cite{konheim1966occupancy}.

Generalizing work of Qiu and Remmel on pattern avoidance in ordered set partitions~\cite{qiu2019patterns}, where they enumerated parking functions which avoided 123, in~\cite{adeniran2023pattern} Adeniran and Pudwell studied and enumerated parking functions based on avoiding certain sets of patterns in an associated permutation obtained by recording the positions of the preferences. For an alternative view of pattern avoidance in parking functions by way of the outcome map, we refer the reader to~\cite{yan2025results}. In particular Adeniran and Pudwell studied and enumerated parking functions which avoided all patterns for every nontrivial subset of $\mathfrak{S}_3$ containing at least two permutations. One collection of problems left open was the enumeration for any $\sigma\in \mathfrak{S}_3\setminus \{123\}$. In this work we provide an answer for all remaining cases.  Additionally, we provide growth rates for parking functions avoiding any monotone pattern of any length. 

The first main result of the paper, in context of existing open problems, is an enumeration of the 321--avoiding parking functions in Theorem~\ref{thm:count}. 
\begin{theorem*}[Theorem~\ref{thm:count}]
	\[\pf_n(321) =\frac{1}{n+1}\cdot \sum_{a = 0}^{\lfloor \frac{n}{2}\rfloor}\frac{n!(2n-a)_{n-a}(n+a-1)_{a}}{(a!(n-2a)!(n-a+1)_a)^2}\]
\end{theorem*}

Additionally we provide the growth rate for words of length $n$ on $[n]$, and consequently for parking functions, avoiding any monotone pattern in Theorem~\ref{thm:decreasing_growth_rate} and Theorem~\ref{thm:increasing_growth_rate}. 
The second main result is a combinatorial enumeration of the number of Sylvester, and consequently \#--Sylvester, classes for words with a fixed packed content in Theorem~\ref{thm:enum_of_132_words}.
\begin{theorem*}[Theorem~\ref{thm:enum_of_132_words}]
The number of Sylvester classes of words whose packed content is $\alpha\vDash n $ with $|\alpha| = k$ is \[\det\bigg[\binom{\displaystyle j+\sum_{\ell = 1}^i\alpha_{k+1-\ell}}{j-(i-1)}\bigg]_{1\le i,j\le k-1}\]	
\end{theorem*}
 Previously in~\cite{novelli2020hopf} a representation theoretic interpretation of this enumeration by way of noncommutative Lagrange inversion, as described in~\cite{novelli2008noncommutative,novelli2020duplicial}, was considered. A consequence of these results is that we will be able to derive an enumeration for the parking functions which avoid a single nonmonotone pattern of size 3. The formal statement of this result is Theorem~\ref{thm:132and213enum}.
  We note that in \cite{baril2025ascent}, the authors answered the open question of Adeniran and Pudwell in the case where $\sigma =132$, but did not phrase it as such, where they found a bijection between parking functions which avoid 132 and intervals in a subposet of the ascent lattice on Dyck paths and enumerated such intervals via generating functions.

The remainder of the paper is structured as follows. In section~\ref{sec:back} we review relevant and necessary background material. In section~\ref{sec:monotone} we prove our results for parking functions avoiding a single monotone pattern of any length. We end in section~\ref{sec:not_mono} by proving an enumeration for Sylvester classes and \#--Sylvester classes of permutations of a fixed multiset allowing us to derive an enumeration for parking functions which avoid a fixed non-monotone pattern of size 3. 

\section{Background}\label{sec:back} In this section we provide an overview of parking functions, pattern avoidance, the RSK correspondence, as well as the Sylvester and \#--Sylvester congruences.
\subsection{Parking Functions}\label{subsec:pf}
We begin with a more formal definition of the classical parking process. Specifically consider the following scenario. 
	\begin{itemize}
	\item A one way street has $n$ labeled spots in order.
\item $n$ cars enter the street, one at a time, each having a preference for a specific parking spot.
 \item When a car enters the street, one of two things can happen:
	 \begin{itemize} \item If their preferred spot is not taken, they park in that spot.
	 \item If their preferred spot is taken, the car parks in the next available spot on the street if possible.
\end{itemize}
\end{itemize}

By encoding the sequence of each car's preferred spot as some function $p:[n]\to[n]$ we have the following formal definition.

\begin{definition}\label{def:parking}
	A sequence of preferences is a \emph{parking function} if at the end of the parking process everyone was able to park.
\end{definition}
Equivalently, a parking function $p:[n]\to [n]$ is a function satisfying that $|p^{-1}([i])| \ge  i$ for all $1\le i \le n$. Note that this is an alternative way of phrasing the increasing rearrangement condition of the introduction. Alternatively, one can view a parking function as a labeled Dyck path, as first introduced by Garsia and Haiman in~\cite{garsia1996natural}. Specifically by using the content condition and considering a Dyck path as a lattice path from $(0,0)$ to $(n,n)$ with steps of the form $(1,0),(0,1)$ which stays below the line $y=x$, by assigning sets of labels to the horizontal runs and listing the label sets in increasing order one recovers an alternative description of parking functions. Specifically the labels of the horizontal run at height $i$ are the listing of cars whose preference is for spot $i$. Associated to each parking function $p$ is the permutation $\pi(p)$ obtained by reading the labels of the horizontal runs in ascending order. We note that this is a variation in convention since these paths are normally expressed to lie above the line $y=x$ and are labelled upon their vertical steps. 
\begin{example}\label{ex:dyckpf}
Below is the Dyck path description of the parking function $23356165$. The associated permutation is $\pi(p)=612384957$.
	
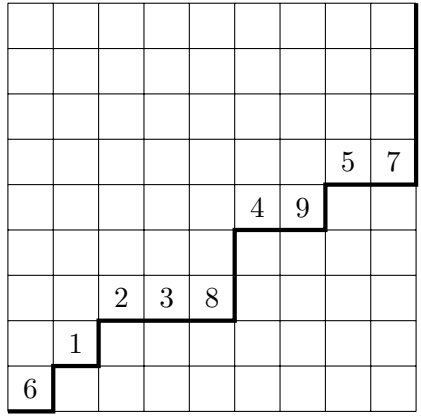
\begin{center}
	\begin{tikzpicture}[scale = .6]
		\draw[step=1.0,black,thin] (0,0) grid (9,9);
		\draw[-, ultra thick] (0,0)--(1,0)--(1,1)--(2,1)--(2,2)--(5,2)--(5,4)--(7,4)--(7,5)--(9,5)--(9,9);
		\node (6) at (0.5,0.5) {6};
		\node (1) at (1.5,1.5) {1};
		\node (2) at (2.5,2.5) {2};
		\node (3) at (3.5,2.5) {3};
		\node (8) at (4.5,2.5) {8};
		\node (4) at (5.5,4.5) {4};
		\node (9) at (6.5,4.5) {9};
		\node (5) at (7.5,5.5) {5};
		\node (7) at (8.5,5.5) {7};

	\end{tikzpicture}
	
\end{center}
\end{example}
A property of parking functions which will be critical is that for each possible set of preferences, i.e. the content of a function $f:[n]\to [n+1]$, there is precisely one cyclic shift of the weak composition describing the content which is the valid content of a parking function. For some exposition better describing this phenomena, we refer to the reader to~\cite{harris2026pollak} for a bijective explanation.
\subsection{Permutations and Pattern Avoidance}
In this subsection we review the necessities of permutations and pattern avoidance. For additional details, see~\cite{bevan2015permutation}. We now begin with the formal definition of patten avoidance in permutations.
\begin{definition}
		Let $\sigma \in \mathfrak{S}_k$. A permutation $\pi \in \mathfrak{S}_n$ is said to \emph{contain} $\sigma$ if there exist $i_1 < i_2 < \cdots < i_k$ such that the relative order of $\pi_{i_1} \pi_{i_2} \dots \pi_{i_k}$ is $\sigma$. A permutation $\pi$ avoids $\sigma$ if $\pi$ does not contain $\sigma$. The collection of all permutations of $\mathfrak{S}_n$ which avoid $\sigma$ is denoted by $\Avoid_n(\sigma)$.
	\end{definition}
	Often times we represent permutations as diagram of $X$'s in a rectangular grid. Specifically, using cartesian coordinates, we place an $X$ in position $(i, j)$ if  $j = \pi(i)$. When we mark the cells $i < j$ for which $\pi(i) > \pi(j)$ we obtain the \emph{Rothe diagram} of a permutation.
	\begin{example}
	Below is the diagram representation of the permutation $45637812$ described above.
		
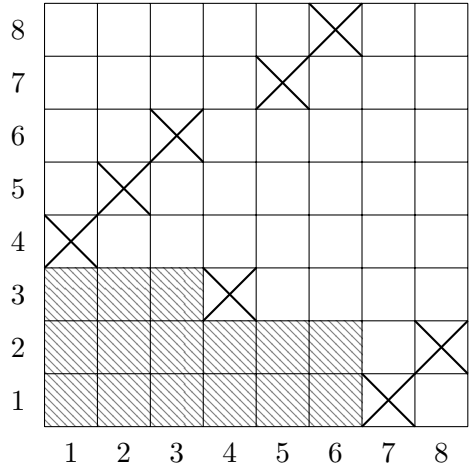
\begin{center}
			\begin{tikzpicture}[scale = .7]
				\lgpermrothe{{4,5,6,3,7,8,1,2}}{8}
			\end{tikzpicture}
		\end{center}
	\end{example}
For an example, the permutation $612384957$ contains 123, but avoids $321$.
	\begin{definition}[\cite{adeniran2023pattern}]
	  For a permutation $\sigma$, we say that a parking function $p$ contains $\sigma$ if $\pi(p)$ contains $\sigma$. Analogously, we say that $p$ avoids $\sigma$ if $\pi(p)$ avoids $\sigma$.  
	  For a set $S$ of permutations, let $\pf_n(S)$ denote the number of parking functions of length $n$ which avoid every permutation $\sigma \in S$.
\end{definition}
Recall that a Cayley word of length $n$ is a word $w_1 w_2 \dots w_n$ over the alphabet of $\mathbb{N}$ with $k = \max\{w_i | i \in [n]\}$ that when viewed as a function $w:[n]\to [k]$ is surjective.
\begin{definition}
	Let $c = c_1c_2\dots c_k$ be a Cayley word. Then a word $w= w_1w_2\dots w_n$ contains $c$ if there is a subword $w'=w_{i_1} w_{i_2}\dots w_{i_k}$ of $w$ such that $w'$ is order isomorphic to $c$. Similarly a word $w$ avoids $c$ if $w$ does not contain $c$.
\end{definition}
\begin{example}
$w=  233561635$ $w$ contains $123$, $w$ contains 312, but $w$ avoids $321$.
\end{example}
\begin{definition}
	Let $C$ be a set of (Cayley) words. Denote by $\PF_n(C)$ the number of parking functions of length $n$ that avoid every $c\in C$.
\end{definition}
It was shown by Jel\'{i}nek and Mansour in~\cite{jelinek2009wilf} that $\PF_n(\sigma)= \PF_n(\sigma')$ for all $\sigma,\sigma'\in \mathfrak{S}_3$. 
\begin{proposition}\label{prop:standardization}
    Let $\pi$ be the permutation obtained by reading the labels of the Dyck path of a parking function $p$. Then $\pi = \std(p)^{-1}$
\end{proposition}

\begin{proof}
    Let $p\in \PF_n$. Viewing $p$ as a labeled Dyck path $D$ with $k$ increasing runs, let $B_i = b_{i_1} \dots  b_{i_{|B_i|}}$ denote the word consisting of the labels of the $i$th increasing run of $D$ listed in increasing order. Then $\pi = B_1 \dots B_k$, that is the concatenation of the $B_i$. Since the labels of each $B_i$ are distinct and as sets partition $[n]$, it follows that $\pi$ is in fact a permutation.

Now, viewing $p$ as a word, let $\sigma = \std(p)^{-1}$. Note that $\sigma$ is then listing in increasing order the indices of the cars that preferred spot $i$ in increasing order of the spots. But this is precisely $\pi$, so $\sigma = \pi$.
\end{proof}

We note that while the above proposition is phrased purely in terms of ordered set partitions corresponding to a parking function, it applies generally to any ordered set partition.
\begin{corollary}\label{cor:dec_stand}
	A parking function $p$ avoids $k+1, k, \dots,1$ if and only if $p$ avoids $k+1,k,\dots, 1$ as a word.
\end{corollary}
\begin{corollary}\label{cor:inc_stand}
	A parking function $p$ avoids $1, 2, \dots , k+1$ if and only if $p$ does not contain a weakly increasing subword of length $k+1$.
\end{corollary}

The last notation for permutations we introduce is the notion of a shuffle. Formally speaking, if $\pi\in\mathfrak{S}_n$ and $\sigma\in\mathfrak{S}_m$, the shuffles of $\pi \shuffle \sigma$ are the permutations $\tau \in \mathfrak{S}_{n+m}$ such that $\pi$ is a subsequence of $\tau$ and $n+\sigma$, where $n+\sigma$ is the sequence $n+\sigma_i$ for $i\in [m]$, is a subsequence of $\tau$. For an example, if $\pi =21$ and $\sigma = 12$ then $\pi\shuffle\tau = \{2134,2314,2341,3214,3241,3412\}$. In particular, the shuffle operation defines the product in the Malvenuto--Reuteneaur algebra~\cite{malvenuto1995duality}. 
\subsection{RSK}
For this subsection and the remainder of the paper we assume the reader is familiar with the notation and theory of partitions. The Robinson--Schensted--Knuth correspondence, hereafter referred to as $RSK$, is a bijection between nonnegative integer matrices and pairs of semistandard Young tableaux of the same shape. More precise details can be found in~\cite[Chapter\ 7]{stanley1999enumerative},~\cite[Chapter\ 4]{fulton1997young}. Recall the Kostka number $K_{\mu \lambda}$ is the number of semistandard tableau of shape $\mu$ and content $\lambda$ and $f^\mu$ is the number of standard young tableau of shape $\mu$. For the purposes of enumeration by using Young's rule suppose we have the collection of words with a fixed content with associated sorted partition $\lambda$. Then the number of words which map under $RSK$ to a fixed shape $\mu$ equals $K_{\mu \lambda} f^{\mu}$.  What will allow us to leverage this for the purposes of enumeration is the following variation on Greene's theorem. For an explicit discussion of how the reference of~\cite{greene1976structure} implies the theorem listed below, we refer the reader to~\cite[Theorem\ 2.1]{garver2019greene} and the surrounding exposition. 

\begin{theorem}[\cite{greene1976structure}]\label{thm:greene}
 Let $w$ be a word of length $n$ over the alphabet $[k]$. Suppose that $RSK(w)= (P,Q)$ where $P$ and $Q$ are of shape $\lambda$. Then the length of the longest decreasing subsequence of $w$ is equal to the number of rows of $\lambda$. Similarly, the longest weakly increasing subsequence of $w$ is equal to $\lambda_1$.
\end{theorem}
Now if we wish to properly enumerate we will need to know $f^\mu$. This count was first given by the classical hook-length formula of Frame, Robinson, and Thrall. 
 Before stating the enumeration we briefly recall the definition of hook lengths and the content of a cell.
 \begin{definition}
	For $u =(i,j)$ a cell in a partition $\lambda$, the content of $u$ is $i-j=c(u)$. The hook length of $u$, denoted by $h(u)$ is the number of cells in the same row and weakly to the right of $u$ or in the same column and weakly below $u$.
\end{definition}

\begin{example}[Hook Contents]
Below in the partition 643 each cell is labeled by its content.
	\[\begin{Young}
	0&1&2&3&4&5\cr
	{-1}&0&1&2\cr
	{-2}&{-1}&0\cr 
\end{Young}\]
\end{example}

\begin{example}[Hook Lengths]
Below in the partition 643 each cell is labeled by its hook length.
\[	\begin{Young}
	8&7&6&4&2&1\cr
	5&4&3&1\cr
	3&2&1\cr 
\end{Young}\]
\end{example}

\begin{theorem}[Hook Length Formula~\cite{frame1954hook}]\label{thm:hooklength}
 Let $f^\lambda$ denote the number of standard Young tableau of shape $\lambda$. Then \[f^\lambda = \frac{n!}{\displaystyle \prod_{u\in \lambda}h(u)}\]
\end{theorem}
The other main representation theoretic enumeration which will be essential for us is the enumeration of semistandard tableaux of shape $\lambda$ with maximum entry $t$. By the definition of the Schur function $s_\lambda$ as a weighted sum over semistandard tableaux of shape $\lambda$, this count is thus just the evaluation where the first $t$ variables are evaluated at 1 and all others at $0$. This quantity has a known enumeration due to Stanley and is known as the hook content formula.
\begin{theorem}[{Hook Content Formula~\cite[Corollary\ 7.21.5]{stanley1999enumerative}}]\label{thm:hookcontent}
The number of semistandard Young tableau of shape $\lambda$ with maximum entry at most $t$ is $s_\lambda(1^t)$. In particular,
  \[s_\lambda(1^t) = \prod_{u\in \lambda}\frac{t+c(u)}{h(u)}.\]
\end{theorem}

\subsection{Sylvester congruences}
We briefly review the Sylvester congruence and the \#--Sylvester congruence. 
\begin{definition}[cf. \cite{hivert2005algebra}]
	Let $w$ be a word. Its binary search tree $P(w)$ is obtained as follows: reading $w$ from right-to-left, one inserts each letter in a binary search tree in the following way: if the tree is empty, one creates a node labeled by the letter; otherwise, this letter is recursively inserted in the left (resp., right) subtree if it is smaller than or equal to (resp., strictly greater than) the root.
\end{definition}
\begin{example}Below are the pair of trees obtained by inserting $155432$ and 554312.

\begin{center}
	\begin{tikzpicture}
			\node (Label) at (-2,-1) {$P(155432)=$};
			\node[circle,draw=black, inner sep=0pt,minimum size=5pt] (root) at (0,0) {2};
			\node[circle,draw=black, inner sep=0pt,minimum size=5pt] (r) at (.5,-.5) {3};
			\draw[-] (root) -- (r);
			
			\node[circle,draw=black, inner sep=0pt,minimum size=5pt] (rr) at (1,-1) {4};
			\draw[-] (rr) -- (r);
			\node[circle,draw=black, inner sep=0pt,minimum size=5pt] (rrr) at (1.5,-1.5) {5};
			\draw[-] (rrr) -- (rr);
			\node[circle,draw=black, inner sep=0pt,minimum size=5pt] (rrrl) at (1,-2) {5};
			\draw[-] (rrr) -- (rrrl);
			\node[circle,draw=black, inner sep=0pt,minimum size=5pt] (l) at (-.5,-.5) {1};
			\draw[-] (root) -- (l);
			\begin{scope}[shift = {(6,0)}]
			\node (Label) at (-2,-1) {$P(554312)=$};
			\node[circle,draw=black, inner sep=0pt,minimum size=5pt] (root) at (0,0) {2};
			\node[circle,draw=black, inner sep=0pt,minimum size=5pt] (r) at (.5,-.5) {3};
			\draw[-] (root) -- (r);
			\node[circle,draw=black, inner sep=0pt,minimum size=5pt] (rr) at (1,-1) {4};
			\draw[-] (rr) -- (r);
			\node[circle,draw=black, inner sep=0pt,minimum size=5pt] (rrr) at (1.5,-1.5) {5};
			\draw[-] (rrr) -- (rr);
			\node[circle,draw=black, inner sep=0pt,minimum size=5pt] (rrrl) at (1,-2) {5};
			\draw[-] (rrr) -- (rrrl);
			\node[circle,draw=black, inner sep=0pt,minimum size=5pt] (l) at (-.5,-.5) {1};
			\draw[-] (root) -- (l);
			\end{scope}

		\end{tikzpicture}

\end{center}
		\end{example}
\begin{definition}[\cite{hivert2005algebra}]
	Let $w_1,w_2 \in [n]^*$. Then $w_1, w_2$ are \emph{Sylvester adjacent} if there exist $u,v,w\in [n]^*$ and $a\le b < c \in [n]$ such that  $w_1 = u a c v b w$ and $w_2 = u c a v b w$.  Two words $u,v$ are said to \emph{Sylvester equivalent} if there is a sequence $u = w_1, w_2, \dots, w_k = v$ such that $w_i$ and $w_{i+1}$ are Sylvester adjacent. We denote this by $u\sylv v$

\end{definition}

\begin{example}
		One can observe that $155432\sylv 554312$ with a sequence of Sylvester adjacencies given as follows:
	$w_1 = 155432, w_2 = 515432, w_3 = 551432, w_4 = 554132, w_5 = 554312$
\end{example}

\begin{theorem}[\cite{hivert2005algebra}]
Let $u,v\in [n]^*$. Then $u\sylv v$ if and only if $P(u) = P(v)$. 
\end{theorem}
In particular, $u\sylv v$ if and only if the contents are equal and $std(u)\sylv std(v)$.

\begin{definition}
	Let $w_1,w_2 \in [n]^*$. Then $w_1, w_2$ are \emph{\#--Sylvester adjacent} if there exist $u,v,w\in [n]^*$ and $a< b \le c \in [n]$ such that  $w_1 = u b  v ac w$ and $w_2 = u b v ca w$.  Two words $u,v$ are said to \emph{\#--Sylvester equivalent} if there is a sequence $u = w_1, w_2, \dots, w_k = v$ such that $w_i$ and $w_{i+1}$ are \#--Sylvester adjacent. We denote this by $u\ssylv v$

\end{definition}

 \begin{example}
 		One can observe that $432115\ssylv 453211$ with a sequence of \#--Sylvester adjacencies given as follows:	$w_1 = 432115, w_2 = 432151, w_3 = 432511, w_4 =435211 , w_5 = 453211$
\end{example}

$u\ssylv v$ if and only if the contents are equal and $std(u)\ssylv std(v)$.

\section{Monotone Patterns}\label{sec:monotone}
In this section we establish a non-closed form for the exact enumeration for parking functions which avoid 321. We note that by using essentially the same techniques one can derive a similar formula for the enumeration for any monotone pattern, that is an increasing or decreasing pattern, of any length. 

Additionally we give the growth rate for any monotone pattern. To do so we will prove the analogous statements for words from which the parking function case will follow. Starting off, let $W_{n,k}$ denote the words of length $n$ whose entries are in $[k]$. For $C$ a collection of Cayley words, let $W_{n,k}(C)$ denote the collection of words $W_{n,k}$ which avoid each $c\in C$ and $w_{n,k}(C) = |W_{n,k}(C)|$. Similarly, when each $c\in C$ is a permutation, let $\overline{W}_{n,k}(C)$ denote the collection of words in $W_{n,k}$ whose standardization avoids each $c\in C$ and $\overline{w}_{n,k}(C)= |\overline{W}_{n,k}(C)|$.
\subsection{Avoiding Decreasing Patterns}
To begin, we formally state a folk theorem that will be essential for our enumeration.
\begin{proposition}\label{prop:rsk_hook_content}
    \[w_{n,k}(r,r-1,\dots, 1) =\sum_{\substack{\lambda\vdash n\\ \ell(\lambda) < r}}s_{\lambda}(1^k) f^\lambda.\]
    \[\overline{w}_{n,k}(1,2,\dots, r) =\sum_{\substack{\lambda\vdash n\\ \lambda_1 < r}}s_{\lambda}(1^k) f^\lambda.\]

\end{proposition}
\begin{proof}
 This statement for $w_{n,k}(r,r-1,\dots, 1)$ immediately since the words which avoid a decreasing subword of length $r$ are precisely those which under $RSK$ map to a pair of tableau of shape $\lambda$ where $\lambda$ has at most $r-1$ parts and the entries are at most $k$. As for the statement for $\overline{w}_{n,k}$, this follows as the length of the first row of $\lambda$ is the length of the longest weakly increasing subword.
\end{proof}
In general the above formula can be computed relatively efficiently, but requires summing over polynomially many terms depending on $r$. When $r=3$ the enumeration simplifies dramatically.
\begin{corollary}\label{thm:word_count}
  \[ w_{n,k}(321) = \sum_{a = 0}^{\lfloor \frac{n}{2}\rfloor}\frac{n!(k+n-a-1)_{n-a}\cdot(k+a-2)_{a}}{(a!(n-2a)!(n-a+1)_a)^2}\]
\end{corollary}
\begin{proof}
    We first observe that if $\lambda\vdash n$, with $\lambda=(n-a,a)$ then $\displaystyle\prod_{u\in \lambda}h(u) = a!\cdot(n-2a)!\cdot (n-a+1)_a$. To see why, first that the set of hook lengths of the second row is $[a]$ and the product of those hook lengths is $a!$. Next, observe that set of hook lengths in the first row of boxes without a box below them is $[n-2a]$ and their product is $(n-2a)!$. Finally, the set of hook lengths in the first row of cells with a box below them is $\{n-a+1, n-a, \dots, n-2a+2\}$ and their product is $(n-a+1)_a$. 

    We now consider the hook contents of $\lambda$, where again $\lambda\vdash n, \lambda=(n-a,a)$. In this case the contents of the first row are $\{0, 1, \cdots, n-1-a\}$ and the contents of the second row are $\{-1, 0, \cdots, a-2\}$. If we then consider the product of the shifts of these contents by $k$ we immediately see that the resulting product is $(k+n-a-1)_{n-a}\cdot(k+a-2)_{a}$ by Theorem~\ref{thm:hookcontent}. By the preceding argumentation and Theorem~\ref{thm:hooklength} $f^{n-a,a} = \frac{n!}{a!(n-2a)!(n+1-a)_a}$. Additionally we have that $s_{n-a,a}(1^k) = \frac{(k+n-a-1)_{n-a}\cdot(k+a-2)_{a}}{a!(n-2a)!(n+1-a)_a}$. Consequently by Proposition~\ref{prop:rsk_hook_content}, we have that \[\displaystyle w_{n,k}(321) =\sum_{a=0}^{\lfloor\frac{n}{2}\rfloor}\frac{n!(k+n-a-1)_{n-a}\cdot(k+a-2)_{a}}{(a!(n-2a)!(n+1-a)_a)^2}.\]
\end{proof}
We are now in a position to answer the first of the enumerative questions raised by Adeniran and Pudwell in~\cite{adeniran2023pattern}. First note that by Corollary~\ref{cor:dec_stand} we have that $\PF_n(321) =\pf_n(321)$.
\begin{theorem}\label{thm:count}
    \[|\PF_n(321)| = \frac{1}{n+1} w_{n,n+1}(321)\]
\end{theorem}
\begin{proof}
The number of $321$ avoiding parking functions are precisely the elements who under $RSK$ give rise to a two row shape by Theorem~\ref{thm:greene}. Since parking functions are $\mathfrak{S}_n$ invariant, then the number of such functions is $\sum_{0\le a\le \lfloor\frac{n}{2}\rfloor} f^{n-a,a}\cdot C_\lambda$, where $C_\lambda$ is the multiplicity of the Specht module $S^\lambda$ in the representation arising from this $\mathfrak{S}_n$ action. As shown by Haiman in~\cite[Equation 28]{haiman1994conjectures}, $C_\lambda=\frac{1}{n+1}s_\lambda(1^{n+1})$. By Proposition~\ref{prop:rsk_hook_content}, this is just $\frac{1}{n+1}w_{n,n+1}(321)$.
\end{proof}
We note that the equality $C_\lambda=\frac{1}{n+1}s_\lambda(1^{n+1})$ will follow from the fact that the content of any parking function is the unique representative of the cyclic rotation of the content of a word $w$ of length $n$ with entries in $[n+1]$. Additionally, one can observe that the exact same argument will apply in the case where a parking function $p$ avoids $k+1,k,\dots, 1$.

Next, we turn our attention to the growth rate of $w_{n,n}(k+1,k,\dots,1)$.
 Previously the asymptotics of words avoiding $321$ was considered in~\cite{branden2005finite}. While we do not recover the same level of precision in the asymptotics, we do recover the growth rate. For readability, let $w_{n,k} = w_{n,n}(k+1,k,\dots, 1)$. 
\begin{theorem}\label{thm:decreasing_growth_rate} Then the growth rate of $w_{n,n}(k+1,k,\dots, 1)$ exists and is $\frac{(k+1)^{k+1}}{k^{k-1}}$. More formally, 
	\[\limsup_{n\to \infty} w_{n,k}^{\frac{1}{n}} = \frac{(k+1)^{k+1}}{k^{k-1}} \]
\end{theorem}
For readability, the main steps of the proof of Theorem~\ref{thm:decreasing_growth_rate} are broken apart into multiple lemmas. First, we establish a pair of technical lemmas which will allow us to show that the desired growth rate exists.

\begin{lemma}
	For a fixed $k$, \[w_{n,k}\cdot w_{m,k} \le w_{n+m,k}\]
\end{lemma}
\begin{proof}
	To establish the result, we show that the natural map of $\iota: W_{n,n}(k+1,k,\dots, 1)\times W_{m,m}(k+1,k,\dots, 1)\to [n+m]^{n+m}$ defined by $(u,v)\mapsto w$ where $w$ is the word on $[n+m]^{n+m}$ \[w_i=\begin{cases}
	u_i & \text{ if }1\le i \le n\\
	v_{i-n}+n & \text{ if } n+1\le i \le n+m
\end{cases}\]
is an injective map to $W_{n+m,n+m}(k+1,\dots, 1)$. First, it is immediate that $\iota(u,v)$ also avoids a decreasing subword of length $k+1$, as each of the first $n$ letters is strictly less than each of the last $m$ letters. That this is injective also follows immediately as we can recover $u$ by restricting to the first $n$ letters of $w$ and $v$ by restricting to the last $m$ letters of $w$ and subtracting $n$ from each entry.
\end{proof}
\begin{lemma} For each $k$, 
	$w_{n,k} \le \binom{2n-1}{n}\cdot |\Avoid_n(k+1,k,\dots, 1)|$.
\end{lemma}
\begin{proof}
	We show this by constructing a surjection $\Pi$ from pairs $(A,\sigma)$ where $A$ is a multiset of $[n]$ of size $n$ and $\sigma$ is a permutation avoiding $k+1, k, \dots ,1$, to words avoiding a decreasing subword of length $k+1$. For such a multiset $A$ and permutation $\sigma$ we define $\Pi(A,\sigma)$ as follows. Let $v$ be the word corresponding to the multiset $A$ written in increasing order. Then $\Pi(A,\sigma)$ is the word $v'$ defined by $v'_i = v_{\sigma(i)}$. To show that $\Pi$ is surjective to $W_{n,n}(k+1, k, \dots, 1)$ suppose that $w\in W_{n,n}(k+1, k, \dots, 1)$ and let $\sigma = \std(w)$. Since $w$ has no decreasing sequence of length $k+1$ if and only if $\std(w)$ has no decreasing sequence of length $k+1$, it follows that $\sigma\in \Avoid_n(k+1,k,\dots, 1)$. Letting $A$ be the multiset that is the content of $w$, it is easily seen then that $\Pi(A,\sigma)=w$. Thus $\Pi$ is a surjection. Since the number of multisets of $[n]$ of size $n$ is $\binom{2n-1}{n}$, the claim follows.
\end{proof}

Importantly, since the growth rate, also referred to as the Stanley--Wilf limit, of $|\Avoid_n(k+1,k,\dots, 1)|$ is $k^2$, as shown by Regev in~\cite{regev1981asymptotic}, and $\binom{2n-1}{n}$ has growth rate 4, it follows that $\limsup_{n\to \infty} w_{n,k}^{\frac{1}{n}}$ is at most $4k^2$. Consequently our desired growth rate will exist by Fekete's Lemma. The next lemma concerns the growth rate of $s_{\lambda^n}(1^n) f^{\lambda^n} $ for a family of partitions $\lambda^n \vdash n$ with $\ell(\lambda^n) \le k$ such that each of sequences of $\{\frac{\lambda^n_i}{n}\}$ converges. In particular, we care about such growth rates as these will all be terms in $w_{n,n}(k+1,k,\dots, 1)$.

\begin{lemma}\label{lem:growth_rate_converging_shape}
	Let $(\lambda^n)$ be a sequence of partitions such that $\lambda^n\vdash n$, $\ell(\lambda^n)\le k$ and $\lim_{n\to \infty} \frac{\lambda^n_i}{n}$ converges to $a_i$ for each $1\le i \le k$. Then \[\lim_{n\to\infty} (s_{\lambda^n}(1^n)f^\lambda)^{\frac{1}{n}} = \prod_{i=1}^k \frac{(1+a_i)^{1+a_i}}{a_i^{2a_i}}\]
\end{lemma}
\begin{proof}
To begin, fixing some $n$, let $HL = \prod_{u\in \lambda^n}h(u)$ denote the denominator of the hook length formula of $f^{\lambda^n}$. Note that for every cell $u$, the number of cells weakly below $u$ is at most $k$ and at least 1, since $\ell(\lambda^n)\le k$. So $\prod_{i=1}^k \lambda^n_i !\le HL \le \prod_{i=1}^k \frac{(\lambda^n_i+k)!}{k!}$. Consequently $1\le \frac{HL}{\prod_{i=1}^k\lambda^n_i!  }\le \prod_{i=1}^k \binom{\lambda_i^n+k}{k} <\binom{n+k}{k}^k$. Since $\binom{n+k}{k}$ is a polynomial of degree $k$, for the purposes of a growth rate, our calculation will remain unchanged by using $\prod_{i=1}^k \lambda_i^n !$ in place of $HL$. By similar logic if $HC = \prod_{i=1}^k \frac{(n+\lambda^n_i-1-i)!}{(n-i)!}$ is the numerator of the hook content formula we can replace $HC$ with $\prod_{i=1}^k \frac{(n+\lambda^n_i)!}{n!}$ for the purposes of computing a growth rate.

 We now consider the growth rate of the sequence
 $\frac{n!\prod_{i=1}^k (n+\lambda^n_i)! }{(n!)^k\prod_{i=1}^k (\lambda_i^n!)^2}$. Applying Stirling's approximation, observe that \[\frac{n!\prod_{i=1}^k (n+\lambda^n_i)! }{(n!)^k\prod_{i=1}^k (\lambda_i^n!)^2}\sim \frac{(\frac{n}{e})^n\prod_{i=1}^k (\frac{n(1+
 \frac{\lambda_{i}^n}{n})}{e})^{n(1+\frac{\lambda_i^n}{n})}}{\frac{n^{kn}}{e^{kn}}\prod_{i=1}^k (\frac{n}{e})^{2\lambda_i} (\frac{\lambda_i^n}{n})^{2\lambda_i^n}}\cdot q(n),\] where $q(n)$ is a function bounded above by a polynomial from the non-exponential terms of Stirling's approximation. Observe that by rearranging terms and since $\lambda^n$ is a partition of $n$ this simplifies to 
 \[\frac{\frac{n^{(k+2)n}}{e^{(k+2)n}}\prod_{i=1}^k (1+
 \frac{\lambda_{i}^n}{n})^{n(1+\frac{\lambda_i^n}{n})}}{\frac{n^{(k+2)n}}{e^{(k+2)n}}\prod_{i=1}^k  (\frac{\lambda_i^n}{n})^{2\lambda_i^n}}\cdot q(n) = \frac{\prod_{i=1}^k (1+
 \frac{\lambda_{i}^n}{n})^{n(1+\frac{\lambda_i^n}{n})}}{\prod_{i=1}^k  (\frac{\lambda_i^n}{n})^{2\lambda_i^n}}\cdot q(n).\]
  Finally \[\lim_{n\to\infty} \bigg(\frac{\prod_{i=1}^k (1+
 \frac{\lambda_{i}^n}{n})^{n(1+\frac{\lambda_i^n}{n})}}{\prod_{i=1}^k  (\frac{\lambda_i^n}{n})^{2\lambda_i^n}}\cdot q(n)\bigg)^{\frac{1}{n}}=\frac{\prod_{i=1}^k (1+a_i)^{1+a_i}}{\prod_{i=1}^k a_i^{2a_i}}\] since $\lim_{n\to\infty} \frac{\lambda_i^n}{n} = a_i$ and $q(n)$ is bounded above by a polynomial.
\end{proof}

For the purposes of determining the growth rate of $w_{n,k}$, and consequently the number of parking functions of length $n$ avoiding a decreasing subword of length $k+1$, the last ingredient will be to determine the maximal value of the quantity $\prod_{i=1}^k \frac{(1+a_i)^{1+a_i}}{a_i^{2a_i}}$ where $0\le a_i \le a_{i-1}$ and $\sum_{i=1}^k a_i = 1$.

\begin{lemma}\label{lem:max_dec_sym}
	For $1\ge a_1 \ge \dots \ge a_k \ge 0$ where $\sum_{i=1}^k a_i=1$ the value of $\prod_{i=1}^k \frac{(1+a_i)^{1+a_i}}{a_i^{2a_i}}$ is maximized when $a_i = \frac{1}{k}$ for all $i$.
\end{lemma}
\begin{proof}
	We show that if $a_j > a_{j+1}$, then the new sequence $b_\ell$ where $b_\ell = a_\ell$ for $\ell\neq j, j+1$ and $b_\ell = \frac{a_j+a_{j+1}}{2}$ for $\ell = j,j+1$  then $\prod_{i=1}^k \frac{(1+b_i)^{1+b_i}}{b_i^{2b_i}} > \prod_{i=1}^k \frac{(1+a_i)^{1+a_i}}{a_i^{2a_i}}$. To do so we consider $\log(\prod_{i=1}^k \frac{(1+b_i)^{1+b_i}}{b_i^{2b_i}})-\log(\prod_{i=1}^k \frac{(1+a_i)^{1+a_i}}{a_i^{2a_i}})$. Note that this simplifies to 
	$
		(2+a_{j}+a_{j+1})\log(1+\frac{a_j+a_{j+1}}{2})-2(a_j+a_{j+1})\log(\frac{a_{j}+a_{j+1}}{2}) - ((1+a_j)\log(1+a_j) + (1+a_{j+1})\log(1-a_{j+1}) - 2 a_j \log(a_j)-2a_{j+1}\log(a_{j+1})$
	since $b_{j}=b_{j+1}= \frac{a_j+a_{j+1}}{2}$. 
	This quantity is positive since the function $(1+x)\log(1+x)-2x\log(x)$ is concave down on the interval of $(0,1]$.
	Consequently, the maximum value of this function is achieved when $a_i = a_j=\frac{1}{k}$ for all $i,j$.
	\end{proof}
We are now in a position to prove our claimed growth rate. 

\begin{proof}[Proof of Theorem~\ref{thm:decreasing_growth_rate}]
	To prove our claimed growth rate, let $\lambda^n$ be a sequence of partitions so that the quantity $s_{\lambda^n}(1^{n}) f^{\lambda^n}$ is maximal amongst partitions of $n$ with at most $k$ parts. Note that we are not assuming that $\lambda^n$ is unique, and if this is not the case for any $n$ we can assume that we have arbitrarily chosen $\lambda^n$ among the maximal partitions. Next consider the sequences $\frac{\lambda^n_i}{n}$ and note that these are bounded sequences taking values in the interval $[0,\frac{1}{i}]$. Since there are only $k$ parts, we can find a sequence of partitions $\lambda^{n_s}$ such that each of the sequences $\frac{\lambda^{n_s}_i}{n_s}$ converges to a number $a_i$. Since this is a subsequence of the maximal partitions, the growth rate of $s_{\lambda^n}(1^{n})f^{\lambda^n}$ is equal to the growth rate of this subsequence, which is $\prod_{i=1}^k \frac{(1+a_i)^{1+a_i}}{a_i^{2a_i}}$ by Lemma~\ref{lem:growth_rate_converging_shape}. Consequently the growth rate of $s_{\lambda^n}(1^{n})f^{\lambda^n} = \prod_{i=1}^k \frac{(1+a_i)^{1+a_i}}{a_i^{2a_i}} \le \frac{(k+1)^{k+1}}{k^{k-1}}$ by Lemma~\ref{lem:max_dec_sym}. To see that this bound is in fact the growth rate, it suffices to consider the sequence of partitions $(n^k)$, that is the partitions of $nk$ whose shape is that of an $k\times n$ rectangle. By Lemma~\ref{lem:growth_rate_converging_shape} the growth rate of the sequence $s_{n^k}(1^{nk})f^{n^k}$ is $\frac{(k+1)^{k+1}}{k^{k-1}}$. Since $s_{n^k}(1^{nk})f^{n^k} < w_{nk,k}$, as it is a term in the sum, then $\frac{(k+1)^{k+1}}{k^{k-1}}$ is also a lower bound for the growth rate. Consequently $\lim_{n\to \infty} w_{n,k}^{\frac{1}{n}} = \frac{(k+1)^{k+1}}{k^{k-1}}$.
	\end{proof}

\subsection{Avoiding Increasing patterns}
We now turn our attention to the case of words, and consequently parking function, which avoid a weakly increasing subword of length $k+1$. Again for readability, let $\overline{w}_{n,k}=\overline{w}_{n,n}(1,2,\dots, k+1)$.
\begin{theorem}\label{thm:increasing_growth_rate} The growth rate of $\overline{w}_{n,n}(1,2,\dots, k+1)$ exists and is $\frac{k^{k+1}}{(k-1)^{k-1}}$. More formally, 
	\[\limsup_{n\to \infty} \overline{w}_{n,k}^{\frac{1}{n}} = \frac{k^{k+1}}{(k-1)^{k-1}}. \]
\end{theorem}
Just as before in the increasing case the main steps of the proof are broken into separate technical lemmas. First, we establish that the enumeration is submultiplicative, which will show that the desired growth rate exists.

\begin{lemma}
	For a fixed $k$, \[\overline{w}_{n,k}\cdot \overline{w}_{m,k} \le \overline{w}_{n+m,k}\]
\end{lemma}
\begin{proof}
	To establish the result, we show that the natural map of $\iota: \overline{W}_n(1, 2, \dots, k+1)\times \overline{W}_m(1,2,\dots,k+1)\to [n+m]^{n+m}$ defined by $(u,v)\mapsto w$ where $w$ is the word on $[n+m]^{n+m}$ \[w_i=\begin{cases}
	m+u_i & \text{ if }1\le i \le n\\
	v_{i-n} & \text{ if } n+1\le i \le n+m
\end{cases}\]
is an injective map to $\overline{W}_{n+m}(1,\dots, k+1)$. First, it is immediate that $\iota(u,v)$ also avoids a weakly increasing subword of length $k+1$, as each of the first $n$ letters is strictly greater than each of the last $m$ letters and both $u$ and $v$ avoid weakly increasing subwords of length $k+1$. That $\iota$ is injective also follows immediately as we can recover $u$ by restricting to the first $n$ letters of $w$ and subtracting $m$ from each entry and $v$ by restricting to the last $m$ letters of $w$.
\end{proof}

Since words avoiding a weakly increasing subword of length $k+1$ are a subset of words avoiding a strictly increasing subword of length $k+1$, for which by Theorem~\ref{thm:decreasing_growth_rate} exists, then the potential growth rate of words avoiding a weakly increasing subword of length $k+1$ is bounded and thus exists by Fekete's Lemma. 
 The next lemma concerns the growth rate of $s_{\lambda^n}(1^n) f^{\lambda^n} $ for a family of partitions $\lambda^n \vdash n$ with $\lambda^n_1 \le k$ such that each of sequences of $\{\frac{(\lambda^n)'_i}{n}\}$ converges. In particular we care about such a growth rate as these will all be terms in $\overline{w}_{n,n}(1,2,\dots, k+1)$.

\begin{lemma}\label{lem:growth_rate_converging_shape_inc}
	Let $(\lambda^n)$ be a sequence of partitions such that $\lambda^n\vdash n$, ${\lambda^n}_1\le k$ and $\lim_{n\to \infty}\frac{(\lambda^n)'_i}{n}$ converges to $a_i$ for each $1\le i \le k$. Then \[\lim_{n\to\infty} (s_{\lambda^n}(1^n)f^\lambda)^{\frac{1}{n}} = \prod_{i=1}^k \frac{1}{(1-a_i)^{1-a_i}a_i^{2a_i}}\]
\end{lemma}
\begin{proof}
To begin, fixing some $n$, as before let $HL = \prod_{u\in \lambda^n}h(u)$ denote the denominator of the hook length formula of $f^{\lambda^n}$. Note that for every cell $u$, the number of cells weakly to the right of $u$ is at most $k$ and at least 1, since $\lambda^n_1\le k$. So $\prod_{i=1}^k (\lambda^n)'_i !\le HL \le \prod_{i=1}^k \frac{((\lambda^n)'_i+k)!}{k!}$. Consequently $1\le \frac{HL}{\prod_{i=1}^k(\lambda^n)'{}_i!}\le \prod_{i=1}^k \binom{(\lambda^n)'_i+k}{k} <\binom{n+k}{k}^k$. Since $\binom{n+k}{k}$ is a polynomial of degree $k$, for the purposes of a growth rate, our calculation will remain unchanged by using $\prod_{i=1}^k (\lambda^n)'_i !$ in place of $HL$. By similar logic if $HC = \prod_{i=1}^k \frac{(n-1+i)!}{(n-(\lambda^n)'_i)!}$ is the numerator of the hook content formula we can replace $HC$ with $\prod_{i=1}^k \frac{n!}{(n-(\lambda^n)'_i)!}$ for the purposes of computing a growth rate.

 We now consider the growth rate of the sequence
 $\frac{(n!)^{k+1}}{\prod_{i=1}^k (n-(\lambda^n)'{}_i)!\prod_{i=1}^k ((\lambda^n)'{}_i!)^2}$. Applying Stirling's approximation, observe that 
 \[\frac{(n!)^{k+1} }{\prod_{i=1}^k (n-(\lambda^n)'_i)!\prod_{i=1}^k ((\lambda^n)'_i!)^2}\sim \frac{(\frac{n}{e})^{kn+n}}{\prod_{i=1}^k (\frac{n}{e})^{2(\lambda^n)'_i} (\frac{(\lambda^n)'_i}{n})^{2(\lambda^n)'_i}\prod_{i=1}^k (\frac{n}{e})^{n-(\lambda^n)'_i}(1-\frac{(\lambda^n)'_i}{n})^{n(1-\frac{(\lambda^n)'_i}{n})}}\cdot q(n)\] where $q(n)$ is a function bounded above by a polynomial from the non-exponential terms of Stirling's approximation. By rearranging terms and since $\lambda^n$ is a partition of $n$ this simplifies to 
 \[\frac{1}{\prod_{i=1}^k (1-\frac{(\lambda^n)'_i}{n})^{n(1-\frac{(\lambda^n)'_i}{n})}\prod_{i=1}^k(\frac{(\lambda^n)'_i}{n})^{2n\frac{(\lambda^n)'_i}{n}}}\cdot q(n).\] Finally \[\lim_{n\to\infty} \bigg(\frac{1}{\prod_{i=1}^k (1-\frac{(\lambda^n)'_i}{n})^{n(1-\frac{(\lambda^n)'_i}{n})}\prod_{i=1}^k(\frac{(\lambda^n)'_i}{n})^{2n\frac{(\lambda^n)'_i}{n}}}\cdot q(n)\bigg)^{\frac{1}{n}}=\frac{1}{\prod_{i=1}^k (1-a_i)^{1-a_i} a_i^{2a_i}} \] since $\lim_{n\to\infty} \frac{\lambda_i^n}{n} = a_i$ and $q(n)$ is bounded above by a polynomial.
\end{proof}

For the purposes of determining the growth rate for $\overline{w}_{n,k}$, and consequently the number of parking functions of length $n$ avoiding a weakly increasing subword of length $k+1$, the last ingredient will be to determine the maximal value of the quantity $\prod_{i=1}^k \frac{1}{(1-a_i)^{1-a_i}a_i^{2a_i}}$ where $0\le a_i \le a_{i-1}$ and $\sum_{i=1}^k a_i = 1$.

\begin{lemma}\label{lem:max_inc_sym}
	For $1\ge a_1 \ge \dots \ge a_k \ge 0$ where $\sum_{i=1}^k a_i=1$ the value of $\frac{1}{\prod_{i=1}^k (1-a_i)^{1-a_i} a_i^{2a_i}}$ is maximized when $a_i = \frac{1}{k}$ for all $i$.
\end{lemma}
\begin{proof}
	We show that if $a_j > a_{j+1}$ the new sequence $b_\ell$ where $b_\ell = a_\ell$ for $\ell\neq j, j+1$ and $b_\ell = \frac{a_j+a_{j+1}}{2}$ for $\ell = j,j+1$  then $\frac{1}{\prod_{i=1}^k (1-b_i)^{1-b_i} b_i^{2b_i}} > \frac{1}{\prod_{i=1}^k (1-a_i)^{1-a_i} a_i^{2a_i}}$. To do so we consider the difference of their logarithms. Note that this simplifies to $(a_{j}+a_{j+1}-2)\log(1-\frac{a_j+a_{j+1}}{2})-2(a_j+a_{j+1})\log(\frac{a_{j}+a_{j+1}}{2}) - ((a_j-1)\log(1-a_j) + (a_{j+1}-1)\log(1-a_{j+1}) - 2 a_j \log(a_j)-2a_{j+1}\log(a_{j+1})$ since $b_{j}=b_{j+1}= \frac{a_j+a_{j+1}}{2}$. This quantity is positive since the function $(x-1)\log(1-x)-2x\log(x)$ is concave down on the interval of $(0,1]$.
	Consequently the maximum value of this function is achieved when $a_i = a_j=\frac{1}{k}$ for all $i,j$.
	\end{proof}
We are now in a position to prove our claimed growth rate. 

\begin{proof}[Proof of Theorem~\ref{thm:increasing_growth_rate}]
	The proof is identical to that of Theorem~\ref{thm:decreasing_growth_rate}, except where Lemma~\ref{lem:growth_rate_converging_shape} and Lemma~\ref{lem:max_dec_sym} are replaced with Lemma~\ref{lem:growth_rate_converging_shape_inc} and Lemma~\ref{lem:max_inc_sym}, and where our lower bound comes from the rectangular partitions $k^n$.
		\end{proof}
\section{Non-Monotone Patterns of size $3$.}\label{sec:not_mono}
In this section we will enumerate the parking functions which avoid a single pattern of length 3 which is not $123$ or $321$. To do so, we will enumerate the equivalences classes of $\sylv$, and consequently of $\ssylv$, for words with a fixed content. We first answer the question of Wilf equivalence using the Adeniran--Pudwell definition of avoidance for parking functions.
\begin{proposition}\label{prop:wilfequivs}
For each $n$ we have the following equalities.
\begin{itemize}
	\item $\pf_n(132) = \pf_n(231)$
	\item $\pf_n(213) = \pf_n(312)$
\end{itemize} 
\end{proposition}
\begin{proof}
	We prove only the first equality because the argument for the second is identical but with the \#--Sylvester congruence instead. First observe that the equivalence classes of the Sylvester congruence can be represented by words whose standardizations avoid $132$. Note that the process for turning a word $w$ into its representative does so by converting instances of 132 in $\std(w)$ into occurrences of 312 in $\std(w')$ where $w'$ is obtaine by a Sylvester adjacency. So the number of words $w$ with $\std(w)$ avoiding 132 and avoiding $312$ are equal. Since a parking function $p$ avoids $132$, or 312, in the Adeniran and Pudwell sense if and only if $\std(p)^{-1}$ avoids 132, or 231, by Proposition~\ref{prop:standardization} our claim then follows.
	\end{proof}
	The remainder of this section is dedicated to proving the following via purely combinatorial methods. 
	\begin{theorem}\label{thm:enum_of_132_words}
The number of Sylvester classes of words whose packed content is $\alpha\vDash n $ with $|\alpha| = k$ is \[\det\bigg[\binom{\displaystyle j+\sum_{\ell = 1}^i\alpha_{k+1-\ell}}{j-(i-1)}\bigg]_{1\le i,j\le k-1}.\]	
\end{theorem}

As a consequence we will derive an enumeration for the number of Sylvester classes and \#--Sylvester classes for any family of words provided the set of contents is known. In particular we can derive the following theorem and provide an answer to the last remaining question of Adeniran and Pudwell. 

\begin{theorem}\label{thm:132and213enum}
The enumeration of parking functions which avoid a single non-monotone pattern in the Adeniran--Pudwell sense is given by
	\[\pf_n(132) =\pf_n(231) =\sum_{\alpha \vDash n} \det\bigg[\binom{\displaystyle j-i+\sum_{\ell = 1}^i\alpha_{\ell}}{j-(i-1)}\bigg]\det\bigg[\binom{\displaystyle j+\sum_{\ell = 1}^i\alpha_{k+1-\ell}}{j-(i-1)}\bigg]\]
		\[\pf_n(213)=\pf_n(312)=\sum_{\alpha \vDash n}\det\bigg[\binom{\displaystyle j-i+\sum_{\ell = 1}^i\alpha_{\ell}}{j-(i-1)}\bigg]\det\bigg[\binom{\displaystyle j+\sum_{\ell = 1}^i\alpha_{\ell}}{j-(i-1)}\bigg]\]
		where in all the above determinants the indices are from $1\le i,j \le k-1$ where $|\alpha|=k$.
\end{theorem}

To prove Theorem~\ref{thm:enum_of_132_words}, from which Theorem~\ref{thm:132and213enum} follows, we will first reduce the enumeration to an equivalent question about enumerating 132 avoiding permutations by their inverse descent sets. By then using the Dyck path arising from the Rothe diagram, we will further be able to reduce our question to the problem of enumerating Dyck paths with a fixed ascent composition. Starting off, consider all words standardized content $\alpha = (\alpha_1,\alpha_2, \dots, \alpha_k)$. Denote the set of all such words by $W_\alpha$. For any $\pi = \std(w)$, where $w\in W_\alpha$, we claim that $\iDes(\pi) \subseteq \{\alpha_1, \alpha_1+\alpha_2, \dots, \sum_{i=1}^{k-1}\alpha_i\} = \alpha(S).$ To see why, notice that $\displaystyle \pi\in \overset{k}{\underset{i=1}{\shuffle}}12\dots \alpha_i$, so the only inverse descents of $\pi$ can occur when the first $i+1$ occurs before the last $i$ in $w$. As an immediate consequence of this observation, since the number of equivalence classes under the Sylvester congruence is the number of 132 avoiding permutations contained within this shuffle, it suffices to count the number of 132 avoiding permutations whose inverse descent set is contained within $\alpha(S)$. To do so, we will utilize the \emph{ascent composition} as well as the \emph{descent composition} of a Dyck path. 

\begin{definition}
	Let $\mathcal{D}$ be a Dyck path of semilength $n$ equaling $U^{\alpha_1}D^{\delta_1}U^{\alpha_2}D^{\delta_2}\dots U^{\alpha_k}D^{\delta_k}$. The \emph{ascent composition} $\alpha(\mathcal{D})=(\alpha_1,\alpha_2,\dots, \alpha_k)$ is the composition of $n$ recording the lengths of consecutive sequences of up steps in $\mathcal{D}$. Similarly we define the \emph{descent composition} of $\D$ to be $\delta(\D) = (\delta_1,\delta_2,\dots, \delta_k)$.
\end{definition}
\begin{example}

Below is a Dyck path $\D$, drawn as a lattice path with steps of the form $(1,1)$ and $(1,-1)$ from $(0,0)$ to $(34,0)$ that never passes below the line $y=0$. For $\D$ one can observe that $\alpha(\D)=(2,2,3,1,3,3,3)$ and $\delta(\D)=(1,2,1,4,2,1,6)$.
	\center
\newcommand\start{circle(.05)}
 \newcommand\up{-- ++(a) \start}
 \newcommand\dn{-- ++(b) \start}
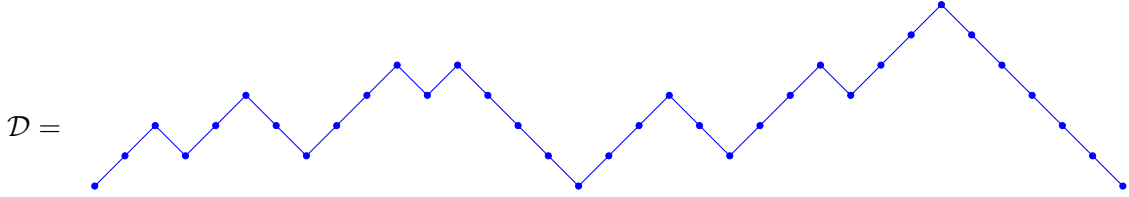
 \begin{tikzpicture}[scale = .8]
 	\coordinate(a) at (.5,.5);
    \coordinate(b) at (.5,-.5);
     \filldraw[blue] (0,0)\start\up\up\dn\up\up\dn\dn\up\up\up\dn\up\dn\dn\dn\dn\up\up\up\dn\dn\up\up\up\dn\up\up\up\dn\dn\dn\dn\dn\dn;
     
			\node (Label) at (-1,1) {$\D=$};
 \end{tikzpicture}
\end{example}
With the prior definition in hand, we can begin to enumerate the number of Sylvester classes of words with a fixed content. Before doing so, we first make the following observation which is formalized in Lemma~\ref{lem:DescentsToAscents}. For $\pi \in \Avoid_n(132)$, consider the Dyck path $\mathcal{D}(\pi)$ corresponding to the boundary of the Rothe diagram of $\pi$. By considering $\mathcal{D}(\pi)$ as a path beginning at $(0,n)$ and ending at $(n,0)$, notice that when $\D(\pi)$ has a valley, that is a pair of consecutive steps of the form $(-1,0),(0,-1)$, then the number of weakly preceding steps of the form $(-1,0)$ corresponds to the position of a descent of $\pi$. This is because of the fact that $\pi$ avoids $132$, so that if there is a new shaded cell in the Rothe diagram of $\pi$ in column $i$ then $\pi(i) > \pi(i+1)$. In particular this also means that the cell of $(i+1,\pi(i+1))$ is a valley. As a consequence we have the following lemma.

\begin{lemma}\label{lem:DescentsToAscents}
	 For any $S\subseteq [n-1]$, we have that \[\#\{\pi \in \Avoid_n(132) | \Des(\pi) = S\} = \#\{\D\in \D_n | \delta(\D) = \alpha(S)\}.\]
\end{lemma}
\newcommand\start{circle(.05)}
 \newcommand\up{-- ++(a) \start}
 \newcommand\dn{-- ++(b) \start}
\begin{example}\label{ex:DyckPathAscentSequenceDescents} 
Below is the 132 avoiding permutation $\pi =345637812$ together with the associated Dyck path $\D(\pi) = UUUUUDDDUDDDUUDD$. Note that $\Des(\pi) = \{3,6\}$and $\alpha(\{3,6\}) = (3,3,2)$. Similarly, observe that $\delta(\D(\pi)) = (3,3,2)$.
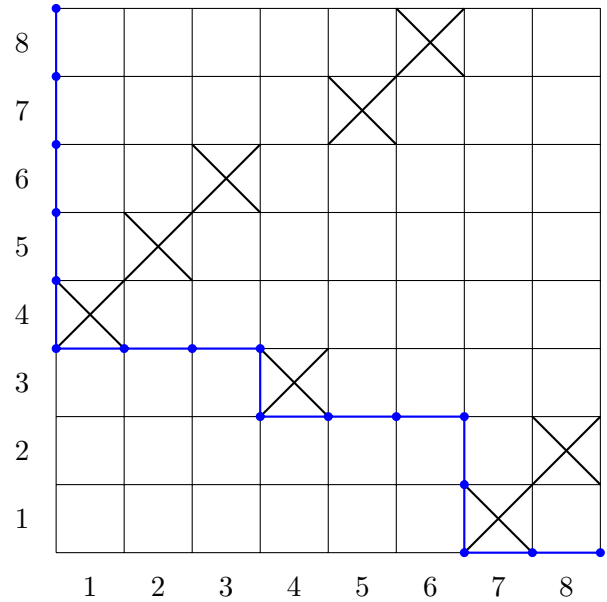
\begin{center}
	\begin{tikzpicture}[scale = .9]
	  \coordinate(a) at (-1,0);
    \coordinate(b) at (0,1);
		\lgpermc{{4,5,6,3,7,8,1,2}}{8}
		  \filldraw[blue, thick] (8,0)\start\up\up\dn\dn\up\up\up\dn\up\up\up\dn\dn\dn\dn\dn;
	\end{tikzpicture}
\end{center}
	\end{example}
Thus, it will suffice to enumerate Dyck paths by their descent composition. It turns out to be easier to do so by their ascent composition. In particular, since if $\pi$ avoids $132$ then so does $\pi^{-1}$ then one can notice that $\D(\pi^{-1})$ is the Dyck path such that $\alpha(\D(\pi^{-1})) = \alpha(\D(\pi))^r$, where $\alpha^r$ denotes the reversal of a composition, since taking the inverse of $\pi$ corresponds to reflecting the Rothe diagram about the line $y=x$.

\begin{lemma}\label{lem:ascent_comp_count}
	Let $\alpha \vDash n$ and suppose $|\alpha|=k$. The number of Dyck paths with ascent composition $\alpha$ is 
		\[\det\bigg[\binom{j-i+\sum_{\ell = 1}^i\alpha_{\ell}}{j-(i-1)}\bigg]_{1\le i,j\le k-1}\]
\end{lemma}
\begin{proof}
	To prove this, we will make use of the classic result of Lindstr\"{o}m--Gessel--Vienot. Let $G$ be the directed graph whose vertex set is $\mathbb{N}\times \mathbb{N}$ with edges directed from $(x,y)\to (x,y+1)$ and $(x,y)\to (x-1,y)$. For a composition $\alpha \vDash n$ with $|\alpha| = k$ let $s_i = (-1+\sum_{\ell=1}^i \alpha_i,i-1)$ and let $t_j = (0,j)$, where $1\le i,j\le k-1$. We claim that the number of Dyck paths with ascent composition $\alpha$ equals the number of families of nonintersecting paths of $G$ directed from $S = \{s_i |i\in [k-1]\}$ to $T =\{t_j | j\in [k-1]\}$. Specifically, we will exhibit an explicit bijection between the two sets of objects. The first thing to observe is that for a family of non-intersecting lattices path from $S$ to $T$ the paths beginning at $s_i$ must terminate at $t_i$. To see why let $P_i$ denote the path that has initial vertex $s_i$. If there is some $i$ where $P_i$ does not terminate at $t_i$ then without loss of generality we can assume that $i$ is minimal. So suppose that the path $P_i$ terminates at $t_j$ with $i\neq j$. Since $i$ is minimal, then $j > i$ meaning there must exist some $i' > i$, in fact $i' = i+1$, such that $P_{i'}$ ends at $t_i$. Because $P_i$ begins below $P_{i'}$ and ends above $P_{i'}$, there must be some vertex where $P_i$ and $P_{i'}$ intersect as they must cross. As a consequence of this, each $P_i$ must contain exactly one edge of the form $(x,y) \to (x,y+1)$.
	
	We now construct a bijection between families of non-intersecting lattice paths from $S$ to $T$ and Dyck paths with ascent composition $\alpha$. To do so, for $i \le k-1$ let $\delta'_i$ equal $(x_i-x_{i-1})$ where in $P_i$ the unique step of the form $(x,y)\to (x,y+1)$ corresponds to the edge $(x_i,i-1)\to (x_i,i)$ and where we take $x_0=-1$. Now let $\delta'(L) = \D'$ be the Dyck path with $\alpha(\D') = \alpha$ and $\delta(\D') = (\delta'_1,\delta'_2,\dots, \delta'_{k-1}, n-\sum_{\ell=1}^{k-1} \delta'_{\ell})$. Our claim is that $\delta'(L)$ is not only well defined but a bijection. 
	
	  We will show $\delta'(L)$ is well defined by demonstrating that for each $i\le k$ that $\sum_{\ell = 1}^i \delta'_i < \sum_{\ell = 1}^i \alpha_i$. Observe for $i\le k-1$ that $\sum_{\ell = 1}^i\delta'_\ell = 1+x_i$. Since each path $P_i$ begins at $(-1+\sum_{\ell = 1}^{i} \alpha_\ell, i-1)$ and horizontal steps decrease the $x$ coordinate then $1+x_i \le  \sum_{\ell = 1}^{i} \alpha_\ell$. To show that $\delta'$ is a bijection, it will be easier to demonstrate injectivity and surjectivity rather than explicitly construct an inverse. For injectivity, since each $P_i$ contains only a single vertical edge if $\delta'(L) = \delta'(L')$ then the $x$--coordinates of the vertical steps of each $P_i$ agree.  This implies that $L'=L$. 
	  
	  For surjectivity, suppose that $\D$ is a Dyck path with $\alpha(\D) = \alpha$ and $\delta(D) = (\delta_1, \delta_2,\dots \delta_k)$. We claim that the collection of paths $L = \{P_i | i \in [k-1] \}$ where $P_i$ is from $s_i$ to $t_i$ with a vertical step $(-1+\sum_{\ell = 1}^i \delta_\ell , i-1)\to(-1+\sum_{\ell = 1}^i \delta_\ell , i)$ is a nonintersecting family of lattice paths from $S$ to $T$ such that $\delta'(L) = \delta$. That $\delta'(L) = \delta$ is immediate from the definition of $\delta'$. What is not quite as immediate is that $L$ is nonintersecting. To see that $L$ is a nonintersecting family, since each $\delta_i > 0$, it follows that $x_i<x_{i+1}$. Thus the vertices with $y$--coordinate $i$ occupied by $P_i$ and $P_{i+1}$ are disjoint and thus the paths are nonintersecting.
	\end{proof}
	\begin{lemma}\label{lem:ascent_refinement}
	The number of Dyck paths with ascent composition $\alpha'=(\alpha'_1,\alpha'_2,\dots ,\alpha'_m)$ where $\alpha'$ refines $\alpha = (\alpha_1,\alpha_2,\dots, \alpha_k)$ is the number of Dyck paths with ascent composition $\beta = (\alpha_1+1,\alpha_2+1,\dots, \alpha_{k-1}+1,\alpha_k)$
\end{lemma}
\begin{proof}
	This follows immediately from the fact that the map from Dyck paths with ascent composition $\beta$ to Dyck paths with ascent composition refining $\alpha$ defined by deleting the first $k-1$ peaks, that is deleting the first $k-1$ instances of $UD$ in the word description, is invertible with inverse given by inserting peaks after the $\sum_{\ell = 1}^i \alpha_\ell$ up steps for each $1\le i \le k-1$.
\end{proof}
\begin{example}
	 Below is nonintersecting path family $L$ corresponding to the path $\D$ of Example~\ref{ex:DyckPathAscentSequenceDescents} with $\alpha(\D)=(2,2,3,1,3,3,3)$ and $\delta(\D) = (1,2,1,4,2,1,6)$ under $\delta'$.
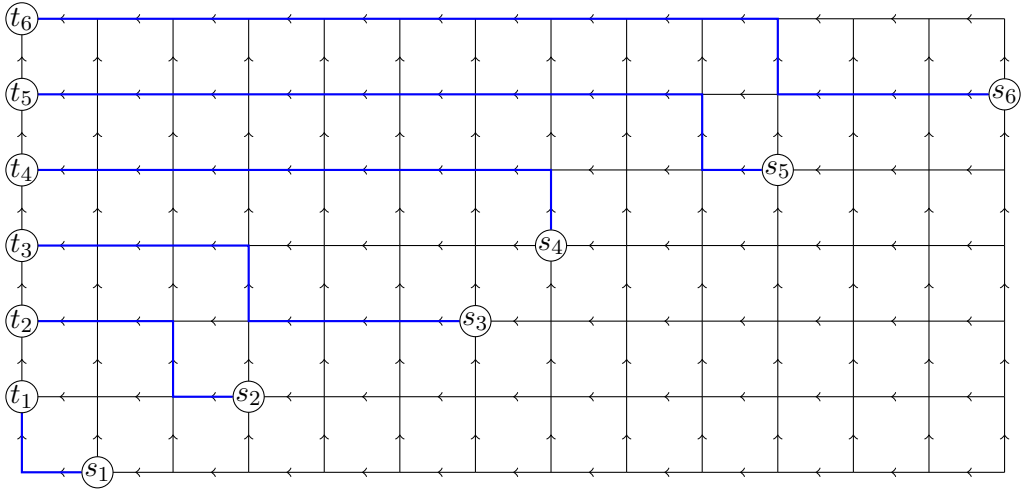
	\begin{center}
	\begin{tikzpicture}[scale = 1]
 	\nwgrid{13}{6}
	\coordinate(a) at (-1,0);
    \coordinate(b) at (0,1);
	\newcommand\w{-- ++(a)}
	\newcommand\n{-- ++(b)}
	\node[circle,draw=black, fill=white, inner sep=0pt,minimum size=5pt] (s1) at (1,0) {$s_1$};
	\node[circle,draw=black, fill=white, inner sep=0pt,minimum size=5pt] (s2) at (3,1) {$s_2$};
	\node[circle,draw=black, fill=white, inner sep=0pt,minimum size=5pt] (s3) at (6,2) {$s_3$};
	\node[circle,draw=black, fill=white, inner sep=0pt,minimum size=5pt] (s4) at (7,3) {$s_4$};
	\node[circle,draw=black, fill=white, inner sep=0pt,minimum size=5pt] (s5) at (10,4) {$s_5$};
	\node[circle,draw=black, fill=white, inner sep=0pt,minimum size=5pt] (s6) at (13,5) {$s_6$};
	\draw[-, thick, blue] (s1)\w\n;
	\draw[-, thick, blue] (s2)\w\n\w\w;
	\draw[-, thick, blue] (s3)\w\w\w\n\w\w\w;
	\draw[-, thick, blue] (s4)\n\w\w\w\w\w\w\w;
	\draw[-, thick, blue] (s5)\w\n\w\w\w\w\w\w\w\w\w;
	\draw[-, thick, blue] (s6)\w\w\w\n\w\w\w\w\w\w\w\w\w\w;
	\node[circle,draw=black, fill=white, inner sep=0pt,minimum size=5pt] (t1) at (0,1) {$t_1$};
	\node[circle,draw=black, fill=white, inner sep=0pt,minimum size=5pt] (t2) at (0,2) {$t_2$};
	\node[circle,draw=black, fill=white, inner sep=0pt,minimum size=5pt] (t3) at (0,3) {$t_3$};
	\node[circle,draw=black, fill=white, inner sep=0pt,minimum size=5pt] (t4) at (0,4) {$t_4$};
	\node[circle,draw=black, fill=white, inner sep=0pt,minimum size=5pt] (t5) at (0,5) {$t_5$};
	\node[circle,draw=black, fill=white, inner sep=0pt,minimum size=5pt] (t6) at (0,6) {$t_6$};
 \end{tikzpicture}
	\end{center}
\end{example}
We can now turn our attention towards proving Theorem~\ref{thm:enum_of_132_words}.
\begin{proof}[Proof of Theorem~\ref{thm:enum_of_132_words}]
	To begin, consider the collection of permutations of the multiset of positive integers $\{\!\{i_1^{\alpha_i},i_2^{\alpha_2},\dots,i_k^{\alpha_k} \}\!\}$ where $i_{1} < i_2 < \cdots < i_k$, each $\alpha_i > 0$, and $\sum_{\ell = 1}^k \alpha_\ell = n$. We denote this collection by $W({\alpha'})$ where the weak composition $\alpha'$ is defined as follows. For each $\ell\le i_k$, $\alpha'_\ell$ is $0$ if $\ell \neq i_j$, for some $1\le j \le k$, and $\alpha'_\ell = \alpha_j$ if $\ell = i_j$.  Each $u\in W_{\alpha'}$ can be expressed as shuffling together $\alpha_j$ copies of $i_j$ for each $1\le j \le k$, so by considering $\sigma = \std(u)$ for $u\in W_{\alpha'}$ we find that $\displaystyle \sigma\in \overset{k}{\underset{j=1}{\shuffle}}12\dots \alpha_j$. In particular we must have for such a permutation $\sigma$ that $\iDes(\sigma)\subseteq \{ \sum_{\ell=1}^i \alpha_i | i\in [k-1]\}  $ and it is immediate this condition characterizes all such standardizations. 
	
	 Recall that for $u,v \in W_{\alpha'}$ we have that $u\sylv v$ if and only if $ \std(u)\sylv \std(v)$ as permutations. Thus the number of Sylvester classes are the number of 132 avoiding permutations which are standardizations of words in $W_{\alpha'}$. Combining this with the fact that $\iDes(\std(u))\subseteq \{ \sum_{\ell=1}^i \alpha_i | i\in [k-1]\}$ and that each Sylvester class contains a unique 132 avoiding permutation, it will suffice to enumerate $\{\sigma \in \Avoid_n(132) | \iDes(\sigma)\subseteq  \{ \sum_{\ell=1}^i \alpha_i | i\in [k-1]\} $. To enumerate these permutations, since $\Avoid_n(132)$ is closed under inversion, we will enumerate $\{\sigma \in \Avoid_n(132) | \Des(\sigma)\subseteq  \{ \sum_{\ell=1}^i \alpha_i | i\in [k-1]\} \}$. Using the descent composition preserving identification of $\Avoid_n(132)$ and $\D_n$ and the fact that $\alpha(\D)^r = \delta(D^r)$, it will suffice to count the number of Dyck paths whose ascent composition $\alpha$ corresponds to a set $S$ which is contained in $\{\sum_{\ell =1}^{i}\alpha^r_i  | i\in[k]\}$, where $\alpha^r$ denotes reversal of the composition $\alpha$. Now by combining Lemma~\ref{lem:ascent_refinement} and Lemma~\ref{lem:ascent_comp_count} this quantity is equal to $\displaystyle\det\bigg[\binom{\displaystyle j-i+\sum_{\ell = 1}^i(\alpha^r_\ell+1)}{j-(i-1)}\bigg]_{1\le i,j\le k-1}=\det\bigg[\binom{\displaystyle j+\sum_{\ell = 1}^i\alpha_{k+1-\ell}}{j-(i-1)}\bigg]_{1\le i,j\le k-1}$ as claimed.
\end{proof}

\begin{example} Below we list the Sylvester classes of the standardization of words with packed content $(2,2,1)$. Additionally we also include below the families of nonintersecting lattice paths for the associated graph for this composition. We wish to stress the color coding is \textbf{\emph{not}} corresponding between Sylvester classes and families of paths.
 \center

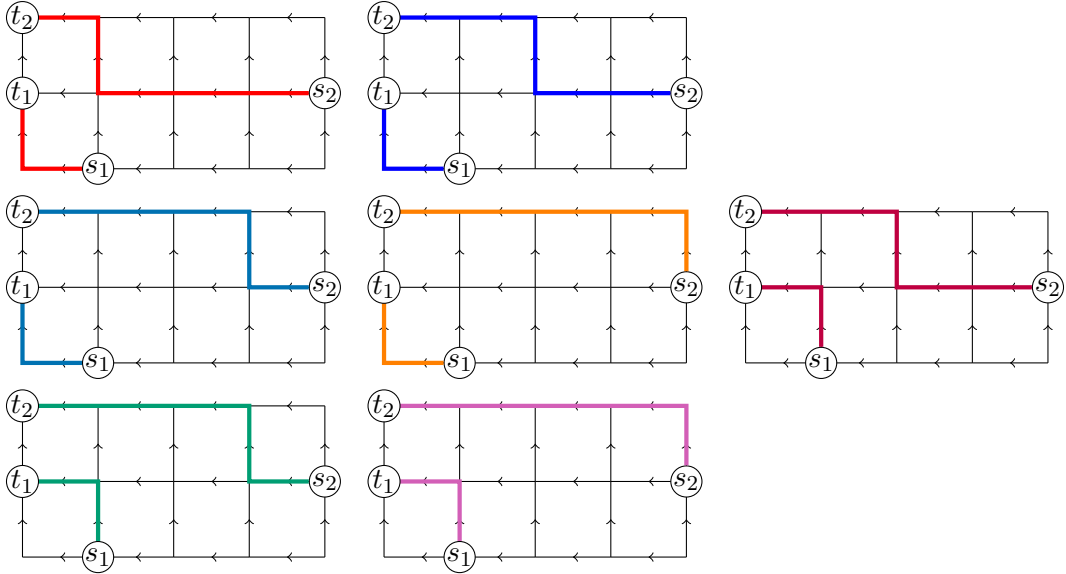
	\begin{tabular}{c c c c c c}
		\color{red}{11223} &\color{red}{12345} &\color{blue}{11232} &\color{blue}{12354} &\color{blue}{11322} &\color{blue}{12534} \\
		\color{orange}{12123} &\color{orange}{13245} &\color{purple}{12132} &\color{purple}{13254} &\color{new_blue}{12213} &\color{new_blue}{13425} \\
		\color{new_pink}{12231} &\color{new_pink}{13452} &\color{purple}{12312} &\color{purple}{13524} &\color{new_pink}{12321} &\color{new_pink}{13542} \\
		\color{blue}{13122} &\color{blue}{15234} &\color{purple}{13212} &\color{purple}{15324} &\color{new_pink}{13221} &\color{new_pink}{15342} \\
		\color{orange}{21123} &\color{orange}{31245} &\color{purple}{21132} &\color{purple}{31254} &\color{new_blue}{21213} &\color{new_blue}{31425} \\
		\color{new_green}{21231} &\color{new_green}{31452} &\color{purple}{21312} &\color{purple}{31524} &\color{new_pink}21321 &\color{new_pink}31542 \\
		\color{new_blue}{22113} &\color{new_blue}{34125} &\color{new_green}{22131} &\color{new_green}{34152} &\color{new_green}{22311} &\color{new_green}{34512} \\
		\color{purple}{23112} &\color{purple}{35124} &\color{new_pink}23121 &\color{new_pink}35142 &\color{new_pink}23211 &\color{new_pink}{35412} \\
		\color{blue}{31122} &\color{blue}{51234} &\color{purple}{31212} &\color{purple}{51324} &\color{new_pink}{31221} &\color{new_pink}{51342} \\
		\color{purple}{32112} &\color{purple}{53124} &\color{new_pink}{32121} &\color{new_pink}{53142} &\color{new_pink}{32211} &\color{new_pink}{53412} \\
	\end{tabular}
	
	\begin{tabular}{c c c }
			\begin{tikzpicture}[scale = 1]
		\nwgrid{4}{2}
		\coordinate(a) at (-1,0);
    	\coordinate(b) at (0,1);
    	\renewcommand\start{circle(.05)}
		\newcommand\w{-- ++(a)}
		\newcommand\n{-- ++(b)}
		\node[circle,draw=black, fill=white, inner sep=0pt,minimum size=5pt] (s1) at (1,0) {$s_1$};
		\node[circle,draw=black, fill=white, inner sep=0pt,minimum size=5pt] (s2) at (4,1) {$s_2$};
		\draw[-, red, ultra thick] (s1)\w\n;
		\draw[-, red, ultra thick] (s2)\w\w\w\n\w;
		\node[circle,draw=black, fill=white, inner sep=0pt,minimum size=5pt] (t1) at (0,1) {$t_1$};
		\node[circle,draw=black, fill=white, inner sep=0pt,minimum size=5pt] (t2) at (0,2) {$t_2$};
	\end{tikzpicture} & 	\begin{tikzpicture}[scale = 1]
		\nwgrid{4}{2}
		\coordinate(a) at (-1,0);
    	\coordinate(b) at (0,1);
    	\renewcommand\start{circle(.05)}
		\newcommand\w{-- ++(a)}
		\newcommand\n{-- ++(b)}
		\node[circle,draw=black, fill=white, inner sep=0pt,minimum size=5pt] (s1) at (1,0) {$s_1$};
		\node[circle,draw=black, fill=white, inner sep=0pt,minimum size=5pt] (s2) at (4,1) {$s_2$};
		\draw[-, blue, ultra thick] (s1)\w\n;
		\draw[-, blue, ultra thick] (s2)\w\w\n\w\w;
		\node[circle,draw=black, fill=white, inner sep=0pt,minimum size=5pt] (t1) at (0,1) {$t_1$};
		\node[circle,draw=black, fill=white, inner sep=0pt,minimum size=5pt] (t2) at (0,2) {$t_2$};
	\end{tikzpicture} & \\
		\begin{tikzpicture}[scale = 1]
		\nwgrid{4}{2}
		\coordinate(a) at (-1,0);
    	\coordinate(b) at (0,1);
    	\renewcommand\start{circle(.05)}
		\newcommand\w{-- ++(a)}
		\newcommand\n{-- ++(b)}
		\node[circle,draw=black, fill=white, inner sep=0pt,minimum size=5pt] (s1) at (1,0) {$s_1$};
		\node[circle,draw=black, fill=white, inner sep=0pt,minimum size=5pt] (s2) at (4,1) {$s_2$};
		\draw[-, new_blue, ultra thick] (s1)\w\n;
		\draw[-, new_blue, ultra thick] (s2)\w\n\w\w\w;
		\node[circle,draw=black, fill=white, inner sep=0pt,minimum size=5pt] (t1) at (0,1) {$t_1$};
		\node[circle,draw=black, fill=white, inner sep=0pt,minimum size=5pt] (t2) at (0,2) {$t_2$};
	\end{tikzpicture} &
	\begin{tikzpicture}[scale = 1]
		\nwgrid{4}{2}
		\coordinate(a) at (-1,0);
    	\coordinate(b) at (0,1);
    	\renewcommand\start{circle(.05)}
		\newcommand\w{-- ++(a)}
		\newcommand\n{-- ++(b)}
		\node[circle,draw=black, fill=white, inner sep=0pt,minimum size=5pt] (s1) at (1,0) {$s_1$};
		\node[circle,draw=black, fill=white, inner sep=0pt,minimum size=5pt] (s2) at (4,1) {$s_2$};
		\draw[-, orange, ultra thick] (s1)\w\n;
		\draw[-, orange, ultra thick] (s2)\n\w\w\w\w;
		\node[circle,draw=black, fill=white, inner sep=0pt,minimum size=5pt] (t1) at (0,1) {$t_1$};
		\node[circle,draw=black, fill=white, inner sep=0pt,minimum size=5pt] (t2) at (0,2) {$t_2$};
	\end{tikzpicture} &
	\begin{tikzpicture}[scale = 1]
		\nwgrid{4}{2}
		\coordinate(a) at (-1,0);
    	\coordinate(b) at (0,1);
    	\renewcommand\start{circle(.05)}
		\newcommand\w{-- ++(a)}
		\newcommand\n{-- ++(b)}
		\node[circle,draw=black, fill=white, inner sep=0pt,minimum size=5pt] (s1) at (1,0) {$s_1$};
		\node[circle,draw=black, fill=white, inner sep=0pt,minimum size=5pt] (s2) at (4,1) {$s_2$};
		\draw[-, purple, ultra thick] (s1)\n\w;
		\draw[-, purple, ultra thick] (s2)\w\w\n\w\w;
		\node[circle,draw=black, fill=white, inner sep=0pt,minimum size=5pt] (t1) at (0,1) {$t_1$};
		\node[circle,draw=black, fill=white, inner sep=0pt,minimum size=5pt] (t2) at (0,2) {$t_2$};
	\end{tikzpicture} \\
	\begin{tikzpicture}[scale = 1]
		\nwgrid{4}{2}
		\coordinate(a) at (-1,0);
    	\coordinate(b) at (0,1);
    	\renewcommand\start{circle(.05)}
		\newcommand\w{-- ++(a)}
		\newcommand\n{-- ++(b)}
		\node[circle,draw=black, fill=white, inner sep=0pt,minimum size=5pt] (s1) at (1,0) {$s_1$};
		\node[circle,draw=black, fill=white, inner sep=0pt,minimum size=5pt] (s2) at (4,1) {$s_2$};
		\draw[-, new_green, ultra thick] (s1)\n\w;
		\draw[-, new_green, ultra thick] (s2)\w\n\w\w\w;
		\node[circle,draw=black, fill=white, inner sep=0pt,minimum size=5pt] (t1) at (0,1) {$t_1$};
		\node[circle,draw=black, fill=white, inner sep=0pt,minimum size=5pt] (t2) at (0,2) {$t_2$};
	\end{tikzpicture} &
	\begin{tikzpicture}[scale = 1]
		\nwgrid{4}{2}
		\coordinate(a) at (-1,0);
    	\coordinate(b) at (0,1);
    	\renewcommand\start{circle(.05)}
		\newcommand\w{-- ++(a)}
		\newcommand\n{-- ++(b)}
		\node[circle,draw=black, fill=white, inner sep=0pt,minimum size=5pt] (s1) at (1,0) {$s_1$};
		\node[circle,draw=black, fill=white, inner sep=0pt,minimum size=5pt] (s2) at (4,1) {$s_2$};
		\draw[-, new_pink, ultra thick] (s1)\n\w;
		\draw[-, new_pink, ultra thick] (s2)\n\w\w\w\w;
		\node[circle,draw=black, fill=white, inner sep=0pt,minimum size=5pt] (t1) at (0,1) {$t_1$};
		\node[circle,draw=black, fill=white, inner sep=0pt,minimum size=5pt] (t2) at (0,2) {$t_2$};
	\end{tikzpicture} &
	\end{tabular}
\end{example}

With the enumeration of the number of Sylvester classes of words with a fixed packed content and the fact for words $u,v$ we have that $u\sylv v$ if and only if $rc(u)\ssylv rc(v)$, then as an immediate consequence we can derive an enumeration for the number of $\ssylv$ classes for a word with packed content $\alpha$.

\begin{corollary}\label{cor:sharpsylvenum}
	The number of equivalence classes of $\ssylv$ for words with packed content $\alpha = (\alpha_1,\alpha_2,\dots, \alpha_k)$ equals $\displaystyle\det\bigg[\binom{\displaystyle j+\sum_{\ell = 1}^i\alpha_{\ell}}{j-(i-1)}\bigg]_{1\le i,j\le k-1}$
\end{corollary}
\begin{proof}
	The number of $\ssylv$ classes of words with packed content $\alpha$ is the number of $\sylv$ classes of words with packed content $\alpha^r$ so the enumeration follows immediately from Theorem~\ref{thm:enum_of_132_words}.
\end{proof}
Corollary~\ref{cor:sharpsylvenum} is the last ingredient needed to prove Theorem~\ref{thm:132and213enum}.
\begin{proof}[Proof of Theorem~\ref{thm:132and213enum}]

We first note that Proposition~\ref{prop:wilfequivs} proves the equalities. 
	Because the argument for 132 and 213 avoiding parking functions is identical, aside from where we use the Sylvester congruence versus the \#--Sylvester congruence, we present only the details for the 132 avoiders. We first note that since for words $u,v$ we have that $u\sylv v$ if and only if the content of $u$ and $v$ are equal and $\std(u)\sylv \std(v)$. This will allow us to decompose our count into summing across the number of Sylvester classes for words of a fixed content. Since each Sylvester class contains a unique element which standardizes to avoid 132 it will suffice to count these elements.
	Further because $\PF_n$ is permutation invariant, then the 132--avoiding parking functions correspond to the Sylvester classes for the words of each valid content. Thus to enumerate the number of Sylvester classes of parking functions, we need only count how many contents there are with a fixed composition for a packed content. Because for each of these compositions the number of Sylvester classes is the same, we need only multiply the number of contents by the number of Sylvester classes for that content. When considering each parking function as a labeled Dyck path, the content of a parking function $p$ is the ascent composition of the associated Dyck path. By combining Lemma~\ref{lem:ascent_comp_count} with Theorem~\ref{thm:enum_of_132_words} we have that the number of parking functions which avoid 132 is enumerated by the quantity \[\sum_{\alpha \vDash n}\det\bigg[\binom{\displaystyle j-i+\sum_{\ell = 1}^i\alpha_{\ell}}{j-(i-1)}\bigg]_{1\le i,j\le k-1}\det\bigg[\binom{\displaystyle j+\sum_{\ell = 1}^i\alpha_{k+1-\ell}}{j-(i-1)}\bigg]_{1\le i,j\le k-1}\] as claimed. 
	\end{proof}
	\section*{Acknowledgments}

This material is based upon work supported by the National Science Foundation under Grant No. DMS-1929284 while the author was in residence at the Institute for Computational and Experimental Research in Mathematics in Providence, RI.
The author thanks ICERM for the opportunity to continue this work in the Collaborate@ICERM program. Additionally the author thanks the developers of OEIS \cite{OEIS} and SageMath \cite{sage}, which were useful in this research, and the CoCalc collaboration platform \cite{SMC}.

The author is also indebted to N\'estor F. D\'iaz Morera, Jennifer Elder, Pamela E. Harris, Molly Lynch, and J. Carlos Mart\'inez Mori who this work grew out of a collaboration with. The author also thanks Nicholas Mayers for helpful feedback on an earlier draft of this paper.
\bibliographystyle{plain}
\bibliography{bib}

@article{bevan2015permutation,
  title={Permutation patterns: basic definitions and notation},
  author={Bevan, David},
  journal={arXiv preprint arXiv:1506.06673},
  year={2015}
}

@article{novelli2020hopf,
  title={Hopf algebras of m-permutations,(m+ 1)-ary trees, and m-parking functions},
  author={Novelli, Jean-Christophe and Thibon, Jean-Yves},
  journal={Advances in Applied Mathematics},
  volume={117},
  pages={102019},
  year={2020},
  publisher={Elsevier}
}

@incollection{novelli2020duplicial,
  title={Duplicial algebras, parking functions, and Lagrange inversion},
  author={Novelli, Jean-Christophe and Thibon, Jean-Yves},
  booktitle={Algebraic combinatorics, resurgence, moulds and applications (CARMA)},
  pages={263--290},
  year={2020},
  publisher={European Mathematical Society-EMS-Publishing House GmbH}
}

@article{novelli2008noncommutative,
  title={Noncommutative symmetric functions and Lagrange inversion},
  author={Novelli, Jean-Christophe and Thibon, Jean-Yves},
  journal={Advances in Applied Mathematics},
  volume={40},
  number={1},
  pages={8--35},
  year={2008},
  publisher={Elsevier}
}

@article{frame1954hook, 
title={The Hook Graphs of the Symmetric Group}, 
volume={6},
DOI={10.4153/CJM-1954-030-1}, 
journal={Canadian Journal of Mathematics}, 
author={Frame, J. S. and Robinson, G. de B. and Thrall, R. M.}, 
year={1954}, pages={316–324}
}

@article{malvenuto1995duality,
title = {Duality between Quasi-Symmetrical Functions and the Solomon Descent Algebra},
journal = {Journal of Algebra},
volume = {177},
number = {3},
pages = {967-982},
year = {1995},
issn = {0021-8693},
doi = {https://doi.org/10.1006/jabr.1995.1336},
url = {https://www.sciencedirect.com/science/article/pii/S0021869385713361},
author = {C. Malvenuto and C. Reutenauer},
}

@article {baril2025ascent,
    AUTHOR = {Baril, Jean-Luc and Bousquet-M\'elou, Mireille and Kirgizov,
              Sergey and Naima, Mehdi},
     TITLE = {The ascent lattice on {D}yck paths},
   JOURNAL = {Electron. J. Combin.},
  FJOURNAL = {Electronic Journal of Combinatorics},
    VOLUME = {32},
      YEAR = {2025},
    NUMBER = {2},
     PAGES = {Paper No. 2.36, 42},
      ISSN = {1077-8926},
   MRCLASS = {05A19 (05A15 06A07 06A11)},
  MRNUMBER = {4908661},
MRREVIEWER = {Joel\ Brewster\ Lewis},
       DOI = {10.37236/13436},
       URL = {https://doi.org/10.37236/13436},
}

@article{qiu2019patterns,
  title={Patterns in words of ordered set partitions},
  author={Qiu, Dun and Remmel, Jeffrey},
  journal={Journal of Combinatorics},
  volume={10},
  number={3},
  pages={433--490},
  year={2019},
  publisher={International Press of Boston}
}

@article{adeniran2023pattern,
    AUTHOR = {Adeniran, Ayomikun and Pudwell, Lara},
     TITLE = {Pattern avoidance in parking functions},
   JOURNAL = {Enumer. Comb. Appl.},
  FJOURNAL = {Enumerative Combinatorics and Applications},
    VOLUME = {3},
      YEAR = {2023},
    NUMBER = {3},
     PAGES = {Paper No. S2R17, 21},
      ISSN = {2710-2335},
   MRCLASS = {05A15 (05A05 05A19)},
  MRNUMBER = {4607273},
MRREVIEWER = {Justin\ M.\ Troyka},
       DOI = {10.54550/eca2023v3s3r17},
       URL = {https://doi.org/10.54550/eca2023v3s3r17},
}

@article{konheim1966occupancy,
author = {Alan G. Konheim and Benjamin Weiss},
title = {An Occupancy Discipline and Applications},
journal = {SIAM Journal on Applied Mathematics},
volume = {14},
number = {6},
pages = {1266-1274},
year = {1966},
}

@article{garsia1996natural,
author={Garsia, A. M. and Haiman, M.}, 
title={Some Natural Bigraded $S_n$-Modules},
volume={3}, 
url={https://www.combinatorics.org/ojs/index.php/eljc/article/view/v3i2r24}, 
DOI={10.37236/1282}, 
number={2},
journal={The Electronic Journal of Combinatorics}, 
year={1996},
month={Jan.},
pages={\#R24} 
}

@article{harris2026pollak,
    author     = {Pamela E. Harris and J. Carlos Martínez Mori and Alexander N. Wilson},
    title      = {A Pollak Proof for the Number of Weakly Increasing Parking Functions},
    url        = {https://dmtcs.episciences.org/17006},
    doi        = {10.46298/dmtcs.17006},
    journal    = {Discrete Mathematics \& Theoretical Computer Science},
    issn       = {1365-8050},
    volume     = {vol. 28:1, Permutation Patterns 2025},
    issuetitle = {Special issues},
    eid        = 2,
    year       = {2026},
    month      = {Apr},
}

@article{greene1976structure,
title = {The structure of sperner k-families},
journal = {Journal of Combinatorial Theory, Series A},
volume = {20},
number = {1},
pages = {41-68},
year = {1976},
issn = {0097-3165},
doi = {https://doi.org/10.1016/0097-3165(76)90077-7},
url = {https://www.sciencedirect.com/science/article/pii/0097316576900777},
author = {Curtis Greene and Daniel J Kleitman},
}

@article{hivert2005algebra,
title = {The algebra of binary search trees},
journal = {Theoretical Computer Science},
volume = {339},
number = {1},
pages = {129-165},
year = {2005},
note = {Combinatorics on Words},
issn = {0304-3975},
doi = {https://doi.org/10.1016/j.tcs.2005.01.012},
url = {https://www.sciencedirect.com/science/article/pii/S0304397505000289},
author = {F. Hivert and J.-C. Novelli and J.-Y. Thibon}
}

@article {haiman1994conjectures,
    AUTHOR = {Haiman, Mark D.},
     TITLE = {Conjectures on the quotient ring by diagonal invariants},
   JOURNAL = {J. Algebraic Combin.},
  FJOURNAL = {Journal of Algebraic Combinatorics. An International Journal},
    VOLUME = {3},
      YEAR = {1994},
    NUMBER = {1},
     PAGES = {17--76},
      ISSN = {0925-9899,1572-9192},
   MRCLASS = {20C30 (05E05)},
  MRNUMBER = {1256101},
MRREVIEWER = {A.\ O.\ Morris},
       DOI = {10.1023/A:1022450120589},
       URL = {https://doi.org/10.1023/A:1022450120589},
}

@article {jelinek2009wilf,
    AUTHOR = {Jel\'inek, V\'it and Mansour, Toufik},
     TITLE = {Wilf-equivalence on {$k$}-ary words, compositions, and parking
              functions},
   JOURNAL = {Electron. J. Combin.},
  FJOURNAL = {Electronic Journal of Combinatorics},
    VOLUME = {16},
      YEAR = {2009},
    NUMBER = {1},
     PAGES = {Research Paper 58, 9},
      ISSN = {1077-8926},
   MRCLASS = {05A18 (05A05)},
  MRNUMBER = {2505100},
MRREVIEWER = {Anna\ E.\ Frid},
       DOI = {10.37236/147},
       URL = {https://doi.org/10.37236/147},
}

@article{regev1981asymptotic,
title = {Asymptotic values for degrees associated with strips of young diagrams},
journal = {Advances in Mathematics},
volume = {41},
number = {2},
pages = {115-136},
year = {1981},
issn = {0001-8708},
doi = {https://doi.org/10.1016/0001-8708(81)90012-8},
url = {https://www.sciencedirect.com/science/article/pii/0001870881900128},
author = {Amitai Regev}
}

@article {yan2025results,
    AUTHOR = {Yan, Jun},
     TITLE = {Results on pattern avoidance in parking functions},
   JOURNAL = {Enumer. Comb. Appl.},
  FJOURNAL = {Enumerative Combinatorics and Applications},
    VOLUME = {5},
      YEAR = {2025},
    NUMBER = {1},
     PAGES = {Paper No. S2R2, 29},
      ISSN = {2710-2335},
   MRCLASS = {05A15 (05A19)},
  MRNUMBER = {4798609},
       DOI = {10.54550/eca2025v5s1r2},
       URL = {https://doi.org/10.54550/eca2025v5s1r2},
}

@article{garver2019greene,
  title={Greene--Kleitman Invariants for Sulzgruber Insertion},
  author={Garver, Alexander and Patrias, Rebecca},
  journal={The Electronic Journal of Combinatorics},
  pages={P3--25},
  year={2019}
}

@article{branden2005finite,
  title={Finite automata and pattern avoidance in words},
  author={Br{\"a}nd{\'e}n, Petter and Mansour, Toufik},
  journal={Journal of Combinatorial Theory, Series A},
  volume={110},
  number={1},
  pages={127--145},
  year={2005},
  publisher={Elsevier}
}

@misc{OEIS,
    Author = {{OEIS Foundation Inc.}},
    Note = {Published electronically at \url{http://oeis.org}},
    Title = {The {O}n-{L}ine {E}ncyclopedia of {I}nteger {S}equences},
    Year = 2023
}

@manual{sage,
 author = {William\thinspace{}A. Stein and others},
 key    = {Sage Math},
 note   = {\url{http://www.sagemath.org}},
 organization  = {The Sage Development Team},
 title  = {{S}age {M}athematics {S}oftware ({V}ersion 9.4)},
 year   = {2022},
}

@manual{SMC,
  Key          = {CoCalc},
  Author       = {{SageMath Inc.}},
  Title        = {CoCalc Collaborative Computation Online},
  note         = {{\tt https://cocalc.com/}},
  Year         = {2022},
}

@book{fulton1997young,
  title={Young tableaux: with applications to representation theory and geometry},
  author={Fulton, William},
  number={35},
  year={1997},
  publisher={Cambridge University Press}
}

@book{stanley1999enumerative,
  title={Enumerative Combinatorics: Volume 2},
  author={Stanley, Richard},
  year={1999},
  publisher={Cambridge University Press}
}

\end{document}